\documentclass[final,leqno,onefignum,onetabnum]{siamltex1213}

\title{Analysis and design of jump coefficients in discrete stochastic diffusion models}

\author{Lina Meinecke, Stefan Engblom, Andreas Hellander, Per L\"{o}tstedt\thanks{Division of Scientific Computing,
			Department of Information Technology,
			Uppsala University, P. O. Box 337, SE-75105 Uppsala, Sweden.
(\email{lina.meinecke@it.uu.se}, {stefane@it.uu.se}, {andreas.hellander@it.uu.se}, {perl@it.uu.se}). 
}}

\usepackage{amssymb}
\usepackage{amsmath}
\usepackage{caption}
\usepackage{subfigure}
\usepackage{multicol}
\usepackage{xcolor}
\usepackage{algorithmicx} 
\usepackage{algpseudocode} 
\usepackage{algorithm} 
\usepackage{listings}
\usepackage{epstopdf}
\usepackage{mathtools}

\DeclareGraphicsExtensions{.eps}
\def\calE{\mathcal{E}}

\def\calO{\mathcal{O}}

\def\calq{\mathcal{Q}}

\def\calT{\mathcal{T}}
\def\calV{\mathcal{V}}
\def\fatA{\mathbf{A}}

\def\fatC{\mathbf{C}}
\def\fatH{\mathbf{H}}
\def\fatD{\mathbf{D}}
\def\tfatD{\tilde{\fatD}}

\def\fate{\mathbf{e}}
\def\fatf{\mathbf{f}}
\def\fatp{\mathbf{p}}
\def\fatI{\mathbf{I}}
\def\fatK{\mathbf{K}}
\def\fatM{\mathbf{M}}
\def\fatn{\mathbf{n}}

\def\fatu{\mathbf{u}}
\def\fatx{\mathbf{x}}

\def\fatS{\mathbf{S}}
\def\tfatS{\tilde{\fatS}}
\def\fatV{\mathbf{V}}
\def\fatv{\mathbf{v}}

\def\fatx{\mathbf{x}}
\def\faty{\mathbf{y}}

\def\exit{\varepsilon}
\def\fatexit{\boldsymbol{\varepsilon}}
\def\fatgamma{\boldsymbol{\gamma}}
\def\tfatgamma{\tilde{\boldsymbol{\gamma}}}

\def\fatmu{\boldsymbol{\mu}}

\def\tD{\tilde{D}}
\def\tS{\tilde{S}}
\def\tgamma{\tilde{\gamma}}

\def\pOmega{\partial\Omega}

\def\pV{\partial \calV}

\def\mE{\mathbb{E}}

\newcommand{\ordo}[1]{{\cal O} \left( #1 \right)}

\renewcommand{\figurename}{Fig.}

\numberwithin{equation}{section}
\numberwithin{table}{section}
\numberwithin{figure}{section}

\newcommand{\rev}[1]{\textcolor{red}{#1}}

\begin{document}
\maketitle

\begin{abstract}
  In computational system biology, the mesoscopic model of reaction-diffusion kinetics is described by a continuous time, discrete space Markov process. To simulate diffusion stochastically, the jump coefficients are obtained by a discretization of the diffusion equation. Using unstructured meshes to represent complicated geometries may lead to negative coefficients when using piecewise linear finite elements. Several methods have been proposed to modify the coefficients to enforce the non-negativity needed in the stochastic setting.
  In this paper, we present a method to quantify the error introduced by that change. We interpret the modified discretization matrix as the exact finite element discretization of a perturbed equation. The forward error, the error between the analytical solutions to the original and the perturbed equations, is bounded by the backward error, the error between the diffusion of the two equations. We present a backward analysis algorithm to compute the diffusion coefficient from a given discretization matrix. The analysis suggests a new way of deriving non-negative jump coefficients that minimizes the backward error.
The theory is tested in numerical experiments indicating that the new method is superior and minimizes also the forward error.
\end{abstract}



\section{Introduction}\label{sec:intro}

The molecular pathways that regulate cellular function are inherently spatial. Cells have a high level of sub-cellular organization, such as a confined nucleus in eukaryotes or membrane bound reaction complexes. Macromolecules are transported by passive diffusion or active transport, driven by molecular motors, between different areas in the cell in order to arrive at the correct location to perform their function. For example, many gene regulatory pathways rely on a cytoplasmic component, where a signal is propagated from the cell membrane to the nucleus, and a nuclear component, where transcription factors bind to DNA to regulate the expression of genes.   

On a \textit{macroscopic} modeling level, the diffusion equation - a partial differential equation (PDE) - is used to describe the time evolution of the concentration of a population of molecules undergoing diffusion. This is a valid model if molecules are abundant. But in cellular regulatory networks, key proteins such as transcription factors are only present in low copy numbers and the deterministic PDE model becomes inaccurate. Experiments \cite{ELSS, KNWBST, MAA1, Met, MunskyNeuertOuden, RajOuden, Swain01102002} and theory \cite{GaMcWaMa, McQuarrie} have shown the importance of accounting for intrinsic noise when modelling cellular control systems.  Consequently, we need spatial stochastic simulation methods, and diffusion in particular is described by a random walk. We can distinguish between two levels of accuracy.

On the \textit{mesoscopic} level we use a discrete Brownian motion to
model the jump process of the molecules. The domain is partitioned
into compartments or voxels. The state of the system is the number of
molecules of each species in each voxel. Molecules can jump between
neighboring voxels and react if they are in the same voxel.  The
probability density function (PDF) for the probability to be in a
state at a certain time satisfies a master equation. If bimolecular
reactions are included there is in general no analytical solution for
the PDF and a numerical solution is difficult to compute due to the
high dimensionality of the state space. Instead, the stochastic
simulation algorithm (SSA) can be used to generate trajectories of the
system. It was first developed by Gillespie \cite{gillespie, GiHePe}
for reactions independent of space. Its efficiency has been improved
in \cite{CaoGilPet1, GibsonBruck} and it is extended to space
dependency with a Cartesian partitioning of the space in
\cite{ElEh04,HFE,IsaacsonPeskin}.  A more accurate description is the
space-continuous \textit{microscopic} level where individual molecules
are followed along their Brownian trajectories. Methods and software
for this approach are found in \cite{AnAdBrAr10, DBOGSK, MCell08,
  SchUllNoe, ZoWo5a}.

In this work, we focus on diffusion at the mesoscopic level. The
probability per unit of time for a molecule to jump from its voxel to
a neighbor is obtained by a discretization of the Laplace operator in
the diffusion equation on the same mesh. A mathematically equivalent
interpretation is that they are obtained from a discretization of the
Fokker-Planck equation for Brownian motion. The resulting matrix is
the generator of a Markov process and all the off-diagonal entries,
representing transition rates, need to be non-negative.  On the
macroscopic level of deterministic PDEs, requiring non-negative jump
coefficients enforces the discrete maximum principle \cite{Thomee,
  Varga}.  To represent the complicated geometries present in cells
(e.g. mitochondria or convoluted membranes), we work with unstructured
meshes, meaning triangular or tetrahedral meshes in 2D and 3D. Many
interesting cellular processes happen on the membranes. Using a vertex
centred discretization allows us to couple diffusion in the bulk to
diffusion on the surface of the domain in a straightforward way. If a
molecule reaches a boundary node we can use the surface mesh for its
2D diffusion while bound to the membrane.

Piecewise linear finite elements on unstructured meshes are used in \cite{EnFeHeLo} to obtain the jump propensities. Mature software exists for discretizing PDEs with the finite element method (FEM) on a given mesh, e.g. \cite{Fenics}. In 2D, mesh generators are usually able to provide high quality meshes \cite{ErUn}, but in 3D the mesh quality decreases and negative off-diagonal elements often appear in the FEM discretization matrix \cite{BrKoKrSo, Kieri, KoKrNe00}. This matrix can then no longer be interpreted as the generator matrix to a Markov process and thus provide transition rates. For our application in stochastic simulations in systems biology, we need to modify this discretization matrix to guarantee non-negative jump coefficients. In \cite{EnFeHeLo} the negative entries are set to zero and the diagonal element is recalculated so that the row sum equals zero. This changes the diffusion speed and leads to errors in, for example, the time a signal needs to propagate from the nucleus to the cell membrane. To address this, we have in previous work \cite{Lotstedt2015} developed a method that preserves mean first passage times. Another approach to obtain non-negative jump coefficients is to use the finite volume method (FVM). But as we will see, the vertex centered FVM scheme does not approximate diffusion more accurately than a filtered FEM  discretization for typical meshes, despite positive coefficients. These methods make the stochastic simulation of diffusion mesh dependent, but this also holds for the accuracy of the numerical solution of the PDE. FEM and FVM coefficients have been modified in \cite{BuErn, GaWu15, ShYu12} to be non-negative but they depend on the PDE solution making them unsuitable for stochastic simulation.

In this paper we analyse the error introduced by modifying the discretization matrix to enforce non-negative jump coefficients. Since the concentration of the species simulated by the SSA converges towards the solution of the diffusion equation, \cite{Kur70, Kur71}, we quantify the error in this deterministic limit. We use backward analysis to find the diffusion equation solved by the new discretization matrix. We study two error estimates: the \emph{backward error} describing the difference in the diffusion in the equations and the \emph{forward error} describing the error in the solutions to the equations.
The analysis suggests a new method to obtain a non-negative discretization by minimizing the backward error.

In Section~\ref{sec:mesodiff} we describe the mesoscopic model and how
the jump coefficients are obtained for unstructured triangular and
tetrahedral meshes. In Section~\ref{sec:ErrorEstimate}, we develop
theory to bound the forward error by the backward error. An algorithm
is provided in Section~\ref{sec:analysis} for calculating the backward
error and then we show how new error minimizing jump coefficients can
be computed in Section~\ref{sec:design}. In the experiments in
Section~\ref{sec:numres}, we analyse the errors numerically, test our
new method, observe that it also minimizes the forward error in
agreement with the estimates in Section~\ref{sec:ErrorEstimate}, and
discuss possibilities for a practical implementation. Final
conclusions are drawn in Section~\ref{sec:concl}.

Vectors and matrices are written in boldface. A vector $\fatu$ has the
components $u_i$ and the elements of a matrix $\fatA$ are
$A_{ij}$. For vectors and matrices $\|\fatu\|_p$ denotes the vector
norm in $\ell_p$ and $\|\fatA\|_p$ its subordinate matrix norm, and
$\|\fatu\|_{\fatA}^2=\fatu^T\fatA\fatu$ for a positive definite
$\fatA$. The derivative of a variable $u$ with respect to time $t$ is
written $u_t$. If $\fatu(\fatx),\, \fatx\in\Omega,$ varies in space,
then $\|\fatu\|_{L^p}^p=\int_\Omega \|\fatu\|^p_p\, d\Omega$.


\section{Mesoscopic model of diffusion}\label{sec:mesodiff}

To model diffusion in discrete space, the domain of interest $\Omega$ is partitioned into non-overlapping voxels $\calV_k, \; k=1,\ldots, N$. 
Each voxel $\calV_j$ has a node or vertex with the coordinates $\fatx_j$ in the interior, see Fig.~\ref{fig:Discretization}. 
If $\calV_i$ and $\calV_j$ are neighbors, the vertices $\fatx_i$ and $\fatx_j$ are connected by the edge $\fate_{ij}$. Molecules in a voxel $\calV_j$ can diffuse by jumps to a
neighboring voxel $\calV_i$ along the edge $e_{ij}$. 

\begin{figure}[H]
\centering
\includegraphics[scale=.25]{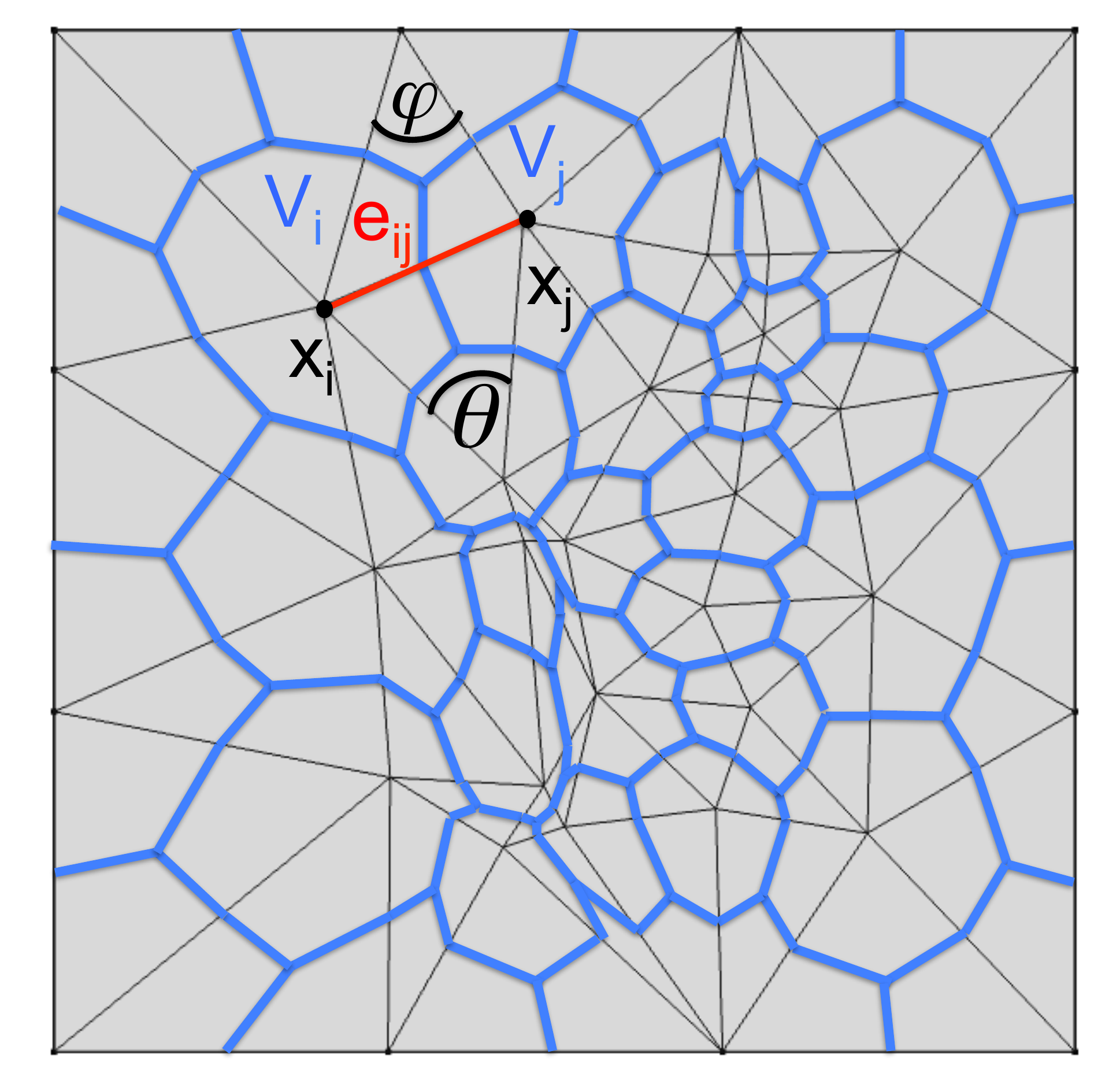} 
\caption{The primal triangular mesh (black), defining the edges $e_{ij}$, and the dual mesh (blue) defining the voxels $\calV_i$.}
\label{fig:Discretization}
\end{figure}

The state of the system is the discrete number of molecules of each species in each voxel.
Let $y_j$ be the number of molecules of chemical species $Y$ in $\calV_j$. 
The jump rate $\lambda_{ji}$ from $\calV_j$ to a neighboring $\calV_i$ needs to satisfy the condition
\begin{equation}
\lambda_{ji}\geq 0.
\label{eq:Positive}
\end{equation}
The total jump rate out of voxel $\calV_j$ is $\lambda_j=\sum_{i,i\neq j}\lambda_{ji}$.
The next time for a jump from $\calV_j$ is exponentially distributed with the
intensity $\lambda_jy_j$. 
Voxel $\calV_i$ is chosen as the destination with a probability proportional to $\lambda_{ji}$. 
After a jump, the number of molecules is updated and the time for a new jump is determined.
This is the Stochastic Simulation Algorithm (SSA) by Gillespie \cite{gillespie} to simulate mesoscopic diffusion of molecules between the voxels. 

\begin{algorithm}[H]
\caption{Stochastic Simulation Algorithm for Diffusion \cite{gillespie}}
\label{alg:SSA}
\begin{algorithmic}[1]
\State Initialize $y_k, k=1,\ldots,N,$ in the $N$ voxels at $t=0$.
\State Sample the exponentially distributed time $t_k$ with rate $\lambda_ky_k$ to the first diffusion event
    in all $N$ voxels.
\State Let $t_j$ be the minimum of all $t_k$. If $t_j> T$ then stop, otherwise continue.
\State For the jump from $\calV_j$, sample a jump to $\calV_i$ with probability $\lambda_{ji}/\lambda_j$.
\State Update $t:=t_j$, $y_i$ and $y_j$. Sample $\Delta t_i$ with the rate $\lambda_iy_i$ and $\Delta t_j$ with the rate $\lambda_jy_j$ and 
    recompute $t_i=t+\Delta t_i$ and $t_j=t+\Delta t_j$. Go to 3.
\end{algorithmic}
\end{algorithm}

We will now show how to derive the propensities $\lambda_{ij}$ from a discretization of the diffusion equation.
Let $\faty$ be a vector with entries $y_i$, describing the number of Y molecules in voxel $\calV_i$.
The probability distribution $p(\faty,t)$ for the distribution of the molecules is the solution to the diffusion-master equation 
\begin{equation}
p_t(\faty,t) = \sum_{i=1}^N\sum_{j=1}^N\lambda_{ij}(\faty-\fatmu_{ij})p(\faty-\fatmu_{ij},t)-\lambda_{ij}(\faty)p(\faty,t),
\end{equation}
where $\mu_{ij,i}=-1$, $\mu_{ij,j}=1$ and $\mu_{ij,k}=0$ for $k\neq i,j$. Calculating the expected value $\overline{y}_i$  of the number of molecules in each voxel $i$ leads to a system of ordinary differential equations (ODEs) for the mean concentration $u_i=\overline{y}_i/|\calV_i|$ in each voxel $\calV_i$
\begin{equation}
u_{it} = \sum_{j=1}^N\frac{|\calV_j|}{|\calV_i|}\lambda_{ji}u_j-u_i\sum_{j=1}^N\lambda_{ij},
\end{equation}
see \cite{EnFeHeLo}. This can be interpreted as a discretization of the diffusion equation
\begin{align}\label{eq:ucont1}
u_t &= \gamma\Delta u = \nabla\cdot(\fatgamma\nabla u), &\fatx\in\Omega,\,t\geq 0,\\
\fatn\cdot\nabla u &= 0, &\fatx\in\partial\Omega,\,t\geq 0,\nonumber\\
u &= u_0, &\fatx\in\Omega,\,t=0,\nonumber
\end{align}
with the diffusion coefficient $\gamma$ and $\fatgamma=\gamma\fatI$.
Thus, the jump rates $\lambda_j$ and $\lambda_{ji}$ can be computed using discretizations of the diffusion equation on the triangular 
or tetrahedral mesh. 
The space derivatives in \eqref{eq:ucont1} are approximated in the voxels by $\fatD$ to obtain equations for the
unknowns $u_i$
\begin{align}
 \label{eq:udisc}
 \fatu_t &= \fatD \fatu.
\end{align}
A discretization with FEM using piecewise linear Lagrangean basis and test functions yields a
mass matrix $\fatM$ and a stiffness matrix $\fatS$. 
The diagonal $\fatA$ is obtained after mass lumping of $\fatM$. 
The diagonal elements are $A_{jj}=|\calV_j|$. Then the system matrix in \eqref{eq:udisc} is
\begin{equation}\label{eq:Ddef}
  \fatD=\fatA^{-1} \fatS.
\end{equation}
Let $h$ be a measure of the mesh size. The solution of \eqref{eq:udisc}
converges to the solution of \eqref{eq:ucont1} when $h\rightarrow 0$ and
the difference between them is of $\calO(h^2)$. 
If the off-diagonal elements $D_{ij}$ in $\fatD$ are non-negative, then these are taken as the jump
coefficients $\lambda_{ji}$ in the SSA in Algorithm~\eqref{alg:SSA} scaled by the volumes of the voxels $|\calV_i|$ and $|\calV_j|$
\begin{equation}
  \displaystyle{\lambda_{ji}=D_{ij}\frac{|\calV_i|}{|\calV_j|}=\frac{S_{ij}}{|\calV_j|}},
\label{eq:lambdadef}
\end{equation}
see \cite{EnFeHeLo}. 
The concentrations $y_j/|\calV_j|$ computed by the SSA converge in the limit of large numbers of molecules
to the concentrations in \eqref{eq:udisc} by \cite{Kur70, Kur71}.

In 2D, the entry $S_{ij}$ corresponding to edge $e_{ij}$ is
\begin{equation}\label{eq:S2D}
S_{ij} = \sin(\varphi+\theta)/2\sin(\varphi)\sin(\theta),
\end{equation}
where $\varphi$ and $\theta$ are the two angles opposing $e_{ij}$ \cite{XuZi}, see Fig.~\ref{fig:Discretization}. If $\varphi+\theta>\pi$ then $S_{ij}<0$ and we can no longer use it to define a jump propensity. A similar condition exists in 3D \cite{XuZi}. Mesh generators in 2D are usually able to construct meshes leading to positive $S_{ij}$ \cite{ErUn}, but in 3D negative off-diagonal entries often occur \cite{Kieri}. The extra requirement in systems biology to have non-negative off-diagonal elements in the stiffness matrix is a sufficient but not necessary condition to fulfill the discrete maximum principle when solving the PDE~\eqref{eq:ucont1} numerically \cite{Thomee, Varga}.

We now present three different methods of modifying the stiffness matrix $\fatS$ or the discretization matrix $\fatD$ containing off-diagonal negative coefficients, so that we can interpret them as the generator matrix of the Markov process simulated by the SSA.
The discretization matrix $\fatD$ is modified to $\tfatD$ in \cite{EnFeHeLo} such that, if $D_{ij}<0$, then 
$\tD_{ij}= 0$ and $\tD_{ii}= -\sum_{j=1}^{n_i} \tD_{ij}$, where $n_i$ is the number of edges leaving vertex $i$. 
This method of calculating 
the jump coefficients by eliminating the negative contributions is denoted here by nnFEM, non-negative FEM.
Convergence of the solution to the equation with the diffusion operator $\gamma\Delta$ is lost but non-negative jump coefficients are defined.
Solving the system of equations
\begin{equation}
    \tilde{\fatu}_{ht} = \fatA^{-1}\tfatS\fatu_h= \tfatD \fatu_h,
\label{eq:udiscdisturb}
\end{equation}
however, can be viewed as a discrete approximation to a perturbed diffusion equation
\begin{equation}
\tilde{u}_t = \nabla\cdot(\tilde{\fatgamma}\nabla\tilde{u}).
\label{eq:ucontdisturb1}
\end{equation}
The diffusion matrix $\tilde{\fatgamma}$ belongs to $\mathbb{R}^{2\times 2}$ in 2D and $\mathbb{R}^{3\times 3}$ in 3D and is symmetric and should be 
positive definite
for all $\fatx$. If $\tilde{\fatgamma}$ is only positive semidefinite it has at least one eigenvalue equal to zero which means that there is no diffusion along the direction of the corresponding eigenvector.
This is unrealistic to happen inside living cells and we do not consider this case, although the following analysis can be generalized to the positive semidefinite case.

Another option is to choose a straightforward finite volume method (FVM).  
If the boundary $\pV_j$ of a voxel $\calV_j$ consists of $n_j$ straight segments (2D) or flat faces (3D)
$\pV_{ji}, i=1,\ldots,n_j,$ 
of length or area $|\pV_{ji}|$ 
with normal $\fatn_{ji}$ of unit length, then 
\begin{equation}
  \displaystyle{\int_{\calV_j} \nabla\cdot(\fatgamma\nabla u)\, dv=
  \int_{\partial \calV_j} \fatn\cdot \fatgamma\nabla u\, ds
            \approx \sum_{i=1}^{n_j} \fatn_{ji}\cdot\fatgamma\,\fate_{ji}(u_i-u_j) \frac{|\pV_{ji}|}{\|\fate_{ji}\|_2^2}}
\label{eq:FVLapl}
\end{equation}
and the stiffness matrix in \eqref{eq:udiscdisturb} is
$\tilde{S}_{ji}=\fatn_{ji}\cdot\fatgamma\,\fate_{ji}{|\pV_{ji}|}/{\|\fate_{ji}\|_2^2}$.
The elements in $\tfatS$ derived from \eqref{eq:FVLapl} are always
non-negative and hence the $\lambda_{ji}$ in \eqref{eq:lambdadef}
defined by the FVM are non-negative.  The $N$ components of $\fatu$
represent the average value of $u$ in the voxels.  However, the
solution of \eqref{eq:udisc} may not converge to the solutions of
\eqref{eq:ucont1} when the mesh size $h$ is reduced \cite{SvGoNo},
since the approximation in \eqref{eq:FVLapl} is consistent with
$\gamma\Delta u$ only if the mesh is of Voronoi type \cite{EyGaHe}.
But we can again interpret $\tilde{D}_{ij}={\tS_{ij}}/{|\calV_i|}$ as
a consistent FEM discretization of the perturbed equation
\eqref{eq:ucontdisturb1}.  The scheme in \eqref{eq:FVLapl} is a vertex
centered FVM. A cell centered FVM is used in \cite{STEPS} to define
the jump coefficients.

In \cite{Lotstedt2015}, the jump coefficients  are chosen to be close to the FEM coefficients in \eqref{eq:Ddef}.
If $D_{ij}$ is non-negative, then $\lambda_{ji}$ is as in \eqref{eq:lambdadef}. 
If $D_{ij}<0$, then the $\lambda_{ji}$ coefficients for voxel $j$ are determined 
such that the mean first exit time $\exit_j$ from a vertex $j$ in the mesh to the boundary $\pOmega$ 
is a solution of the system of linear equations
\begin{equation}
     \tfatD\fatexit =-\fate,\quad \fate^T=(1,1,\ldots,1),
\label{eq:ET}
\end{equation}  
where $\tD_{ij}=\lambda_{ji}|\calV_j|/|\calV_i|$. The mean first exit time is the expected time it takes for a molecule initially at a position
inside $\Omega$ to reach $\pOmega$. With these coefficients in Algorithm~\ref{alg:SSA}, 
the average of the simulated first exit times from vertices in the
mesh agree very well in \cite{Lotstedt2015} with the ones computed numerically with a FEM discretization of \eqref{eq:ET}. 
This method of computing the jump coefficients is based on the global first exit time of the molecules and is denoted by GFET.


\section{Analysis}
\label{sec:ErrorEstimate}

Previously, we presented three methods that can be regarded as
modifications of the FEM discretization matrix $\fatD$ into a matrix
$\tfatD$ with non-negative jump coefficients. In this section we view
$\tfatD$ as a FEM discretization of a certain perturbed PDE
\eqref{eq:ucontdisturb2} below. To quantify the error introduced by
the change in the diffusion matrix we therefore aim at bounding the
difference between the solutions to the PDEs
\begin{align}
 \label{eq:ucont2}
 u_t &=
 \gamma\Delta u,\\
 \label{eq:ucontdisturb2}
 \tilde{u}_t &=
 \nabla\cdot(\tfatgamma\nabla\tilde{u}),
\end{align}
for $\fatx \in \Omega$, with homogeneous Neumann boundary conditions
$\partial u/\partial n = \partial \tilde{u}/\partial n = 0$ for $\fatx \in
\partial \Omega$,
and initial data $u_0 = \tilde{u}_0$ at $t=0$. Here
$\tfatgamma(\fatx)$ is a symmetric, \emph{uniformly} positive definite matrix.  
The mean first exit time used to define the GFET
algorithm fulfills Poisson's equation, \cite{Oksendal}
\begin{align}
 \label{eq:Econt}
 -1 &= \gamma\Delta \exit, 
 \end{align}
 with the corresponding perturbed equation
 \begin{align}
 \label{eq:Econtdisturb}
 -1 &= \nabla\cdot(\tilde{\fatgamma}\nabla\tilde{\exit})
\end{align}
and homogeneous Dirichlet boundary condition.

Let $H^{1}(\Omega)$ be the Hilbert space of all functions $u\in
L^2(\Omega)$ for which the first weak derivative exists and lies in
$L^2(\Omega)$.  The corresponding weak problems for \eqref{eq:ucont2}
and \eqref{eq:ucontdisturb2} are \emph{find $u, \tilde{u} \in
  H^{1}(\Omega)$ such that for $\forall v \in H^{1}(\Omega)$},
\begin{align}
 \label{eq:uweak}
 (v,u_{t}) &= -(\nabla v,\gamma\nabla u), \\
 \label{eq:uweakdisturb}
 (v,\tilde{u}_{t}) &= 
 -(\nabla v,\tfatgamma\nabla \tilde{u}).
\end{align}

For finite element solutions in a finite dimensional subspace
$H^{1}_{h}(\Omega) \subset H^{1}(\Omega)$ we have the set of 
equations (see also \eqref{eq:udisc} and \eqref{eq:udiscdisturb}
above)
\begin{align}
 \label{eq:udisc2}
 \fatM \fatu_{t} &= \fatS \fatu,\\
 \label{eq:udiscdisturb2}
 \fatM \tilde{\fatu}_{t} &= \tfatS \tilde{\fatu}.
\end{align}

We shall require the following basic \textit{a priori} estimates

\begin{lemma}
  \label{lemma:apriori}
  For some $C > 0$,
  \begin{align}
    \label{eq:apriori1_L2}
    \|\nabla u\|_{L^2} &\le \|\nabla u_{0}\|_{L^2}, \\
    \label{eq:apriori2}
    \|\nabla \tilde{u}\|_{L^2} &\le C\|\nabla u_{0}\|_{L^2}, 
  \end{align}
  where \eqref{eq:apriori2} assumes that $\tfatgamma$ is uniformly
  positive definite such that
  $g\|\faty\|_2^2 \leq \faty^T\tilde{\gamma}(\fatx)\faty\leq
  G\|\faty\|_2^2$,
  for some positive constants $(g,G)$,
  $\forall\fatx\in\Omega \subset \mathbb{R}^{d}$,
  $\forall\faty\in\mathbb{R}^d$, where $d$ is the dimension.
\end{lemma}

\begin{proof}
  The case \eqref{eq:apriori1_L2} of a scalar $\gamma$ follows
  straightforwardly so we focus on \eqref{eq:apriori2}. Letting $v =
  \tilde{u}_{t}$ in \eqref{eq:uweakdisturb} we arrive at
  \begin{align*}
    \|\tilde{u}_{t}\|_{L^2}^{2} &= 
    -(\nabla \tilde{u}_{t},\tfatgamma\nabla \tilde{u})
    = -\frac{1}{2} \frac{d}{dt} 
    (\nabla \tilde{u},\tfatgamma\nabla \tilde{u}), \\
    \intertext{since $\tfatgamma$ is symmetric. Integrating we get}
    (\nabla \tilde{u},\tfatgamma\nabla \tilde{u}) &\le 
    (\nabla \tilde{u}_{0},\tfatgamma\nabla \tilde{u}_{0}).
  \end{align*}
  Invoking the definiteness of $\tfatgamma$ we arrive at
  \begin{align*}
    g\|\nabla \tilde{u}\|_{L^2}^{2} &\le G\|\nabla \tilde{u}_{0}\|_{L^2}^{2}.
  \end{align*}  
\end{proof}

We now bound the error in the two solutions $u$ and $\tilde{u}$.
\begin{theorem}
\label{thm:BoundUInfty}
 For some constant $C > 0$,
 \begin{align}
   \|\tilde{u}-u\|_{L^2}^{2} &\le C t \|\fatgamma-\tfatgamma\|_{\infty}
   \|\nabla u_{0}\|_{L^2}^{2},
   \label{eq:FEUInfty}
 \end{align}
 where the norm of the difference of the diffusion rates is defined
 by
 \begin{align}
   \|\fatgamma-\tilde{\fatgamma}\|_{\infty} & \vcentcolon =
   \max_{x\in\Omega}\|\fatgamma-\tilde{\fatgamma}\|_2.
   \label{eq:GlobalErrorInfty}
 \end{align}
\end{theorem}

\begin{proof}
 Subtracting \eqref{eq:uweak} from \eqref{eq:uweakdisturb} and using
 $v = \tilde{u}-u$ we get
 \begin{align*}
   \frac{1}{2} \frac{d}{dt} \|v\|_{L^2}^{2} &= -\gamma\|\nabla v\|_{L^2}^{2}+
   (\nabla v,(\fatgamma-\tfatgamma)\nabla \tilde{u})
   \leq \int_{\Omega} |\nabla v^T(\fatgamma-\tfatgamma)\nabla\tilde{u}|d\Omega\\
   &\leq \int_{\Omega}\|\fatgamma-\tfatgamma\|_2\|\nabla v\|_2\|\nabla\tilde{u}\|_2d\Omega
   \leq \max_{x\in\Omega}\|\fatgamma-\tfatgamma\|_2\int_{\Omega}\|\nabla v\|_2\|\nabla\tilde{u}\|_2d\Omega\\
   &\leq\|\fatgamma-\tfatgamma\|_{\infty}\|\nabla v\|_{L^2} \|\nabla\tilde{u}\|_{L^2}  \leq \|\fatgamma-\tfatgamma\|_{\infty}(\|\nabla u\|_{L^2}+\|\nabla\tilde{u}\|_{L^2}) \|\nabla\tilde{u}\|_{L^2}  \\
   &\leq C\|\fatgamma-\tfatgamma\|_{\infty}\|\nabla u_{0}\|_{L^2}^{2}
\end{align*}
using Lemma \ref{lemma:apriori}. 
The estimate \eqref{eq:FEUInfty}
follows by integration of the inequality.
\end{proof}

This shows that the forward error $\|u-\tilde{u}\|_{L^2}$ is bounded.
Using the maximum norm of the difference between the two diffusion
constants as in \eqref{eq:FEUInfty} is, however, pessimistic and we
now instead use the mean value of $\|\fatgamma-\tfatgamma(\fatx)\|_2$
over $\Omega$ to bound the error in the solutions.
\begin{proposition}
\label{thm:BoundU}
 For some constant $C > 0$,
 \begin{equation}
   \frac{d}{dt} \|\tilde{u}-u\|_{L^2}^{2} \le C \|\fatgamma-\tfatgamma\|_\ast(\|\nabla u\|_{L^4}+\|\nabla\tilde{u}\|_{L^4})\|\nabla\tilde{u}\|_{L^4},
   \label{eq:FEU}
 \end{equation}
 where
 \begin{align}
   \|\fatgamma-\tilde{\fatgamma}\|_\ast^2 & \vcentcolon=
   \frac{1}{|\Omega|}\int_{\Omega}\|\fatgamma -\tilde{\fatgamma}\|_2^2\,d\Omega.
   \label{eq:GlobalError}
 \end{align}
\end{proposition}
\begin{proof}
 As previously subtracting \eqref{eq:uweak} from \eqref{eq:uweakdisturb} and using
 $v = \tilde{u}-u$ we get
 \begin{align*}
   \frac{1}{2} \frac{d}{dt} \|v\|_{L^2}^{2} 
   &\leq \int_{\Omega}\|\fatgamma-\tfatgamma\|_2\|\nabla v\|_2\|\nabla\tilde{u}\|_2d\Omega
   \leq \left(\int_{\Omega}\|\fatgamma-\tfatgamma\|_2^2d\Omega\right)^\frac{1}{2}\left(\int_{\Omega}\|\nabla v\|_2^2\|\nabla\tilde{u}\|_2^2d\Omega\right)^\frac{1}{2}\\
   &\leq
   \|\fatgamma-\tfatgamma\|_\ast|\Omega|^\frac{1}{2}\left(d\int_{\Omega}\|\nabla
     v\|_4^2\|\nabla\tilde{u}\|_4^2d\Omega\right)^\frac{1}{2}
   \leq \|\fatgamma-\tfatgamma\|_\ast|\Omega|^\frac{1}{2}d^{\frac{1}{2}}\|\nabla v\|_{L^4}\|\nabla\tilde{u}\|_{L^4}.
\end{align*}
\end{proof}

For $t\geq\delta>0$ the diffusion equations \eqref{eq:ucont2} and \eqref{eq:ucontdisturb2} smooth out irregularities in the initial data, so that
we can assume
$\|\nabla u\|_{L^4}$ and $\|\nabla\tilde{u}\|_{L^4}$ are bounded, such
that for some $C>0$
\begin{align}
	\label{eq:L4Bound}
	\|\nabla u\|_{L^4}\leq C \quad \text{and} \quad \|\nabla\tilde{u}\|_{L^4}\leq C.
\end{align}
and hence by integration of \eqref{eq:FEU}
\begin{align}
	\label{eq:BoundUFinal}
	\|\tilde{u}-u\|_{L^2}^{2} &\le C t \|\fatgamma-\tfatgamma\|_\ast.
\end{align}
In summary, Theorem~\ref{thm:BoundUInfty} shows that the forward error
can be bounded in terms of the difference
$\|\fatgamma-\tfatgamma\|_{\infty}$ and some factors that are
independent of $\tfatgamma$ (assuming that $\tfatgamma$ is uniformly
positive definite). Equation~\eqref{eq:BoundUFinal} shows a sharper bound in
terms of $\|\fatgamma-\tfatgamma\|_{\ast}$ at the cost of factors that
may depend on $\tfatgamma$. 
On balance we take
\rev{~\eqref{eq:BoundUFinal}} as the basis for our further analysis, thus
assuming essentially that $C$ in \eqref{eq:L4Bound} depends only
mildly on $\tfatgamma$.

The following proposition proves that $u$ and $\tilde{u}$ have
identical steady state, which shows that the $t$-dependent estimates
in Theorem \ref{thm:BoundUInfty} and \eqref{eq:BoundUFinal} are pessimistic
and give a relevant bound only for small $t$.
\begin{proposition}
 \label{prop:eq_error}
 For $t\to\infty$ the steady state solutions of \eqref{eq:ucont2} and
 \eqref{eq:ucontdisturb2} fulfill
 \begin{align}
   \label{eq:SteadyState}
   u_{\infty} = \tilde{u}_{\infty} &= \|u_0\|_{L_1}/|\Omega|.
 \end{align}
\end{proposition}

\begin{proof}
 Using $v = \tilde{u}_{\infty}$ in \eqref{eq:uweakdisturb} we have at the steady
 state that
 \begin{align}
   0 &= (\nabla\tilde{u}_{\infty},\tfatgamma\nabla\tilde{u}_{\infty}).\nonumber
 \end{align}
 By the positive definiteness of $\tfatgamma$ this means
 $\nabla\tilde{u}_{\infty} = 0$ and hence $\tilde{u}_{\infty}$ is
 constant. A similar argument for $u$ implies the same property and,
 moreover, since $u$ is a density we can safely assume $u_ {0} =
 \tilde{u}_{0} \geq 0$. Setting $v=1$ in \eqref{eq:uweakdisturb} we conclude
 \begin{align*}
 	\tilde{u}_{\infty}|\Omega| &= \int_{\Omega}\tilde{u}_{\infty}d\Omega = (\tilde{u}_{\infty},1) = (\tilde{u}_0,1) = \|\tilde{u}_0\|_{L^1}
 \end{align*}
 and analogously for $u_{\infty}$.
\end{proof}

Using $u_0=\tilde{u}_0$ we have from Proposition~\ref{prop:eq_error}
that, since $\|u-\tilde{u}\|_{L^2}$ is continuous in time, there exists a
$t^{\star}\in(0,\infty)$ where the error reaches its maximum
$\|u(t^{\star})-\tilde{u}(t^{\star})\|_{L^2}\geq\|u(t)-\tilde{u}(t)\|_{L^2},\, \forall
t$.

We obtain a similar result for the error in Poisson's equations.
\begin{theorem}
Assume $\partial\Omega\in C^\infty$ and $\tfatgamma(x)\in C^\infty$, then for the weak solutions $\exit$ and $\tilde{\exit}$ of the weak problems corresponding to \eqref{eq:Econt} and \eqref{eq:Econtdisturb} 
 \begin{align}
   \|\exit-\tilde{\exit}\|_{L^2}^2 &\leq C\|\fatgamma-\tilde{\fatgamma}\|_\ast
   \label{eq:FEE}
 \end{align}
for some constant $C>0$.
\end{theorem}

\begin{proof}
 Choosing $v=\exit-\tilde{\exit}$ and subtracting the weak formulations for
 \eqref{eq:Econt} and \eqref{eq:Econtdisturb} we obtain, with $\|\fatv\|^2_{\fatgamma} = (\fatv,\fatgamma\fatv)$,
 \begin{align*}
   0 &= \|\nabla v\|^2_{\tilde{\fatgamma}}+
   (\nabla v,(\fatgamma-\tilde{\fatgamma})\nabla \exit),
 \end{align*}
 using the same arguments as in the proof of Proposition \ref{thm:BoundU} gives
 \begin{align*}
   \|\nabla v\|_{\tilde{\fatgamma}}^2 
   &\leq C\|\fatgamma-\tilde{\fatgamma}\|_\ast\|\nabla v\|_{L^4}\|\nabla \exit\|_{L^4}.
 \end{align*}
 By the Poincar\'{e}-Friedrich's inequality and the positive-definiteness of $\tfatgamma$
 \begin{align*}
   \|v\|_{L^2} &\leq Cg^{-1}\|\fatgamma-\tilde{\fatgamma}\|_\ast\|\nabla v\|_{L^4}\|\nabla \exit\|_{L^4}
 \end{align*}
 By \cite[Chap.~6.3, Theorem~6]{Evans} and the assumptions we obtain
 $\exit\in C^\infty(\bar{\Omega})$ and
 $\tilde{\exit}\in C^\infty(\bar{\Omega})$. This bounds both
 $\|\nabla\exit\|_\infty$ and $\|\nabla\tilde{\exit}\|_\infty$ and
 hence $\|\nabla\exit\|_{L^4}\leq C$ and
 $\|\nabla\tilde{\exit}\|_{L^4}\leq C$.
\end{proof}

We conclude that we can effectively bound both the forward error
$\|u-\tilde{u}\|_{L^2}$ and the error in the mean first exit time
$\|\exit-\tilde{\exit}\|_{L^2}$ by the difference
$\|\fatgamma-\tilde{\fatgamma}\|_\ast$. In the following section we
present an algorithm for calculating this quantity for a given
discretization matrix $\tilde{\fatD}$.


\section{Backward analysis}\label{sec:analysis}

In Section~\ref{sec:mesodiff}, we presented the FVM and the modified
FEM to compute the stiffness matrix $\tfatS$ leading to non-negative
jump coefficients in \eqref{eq:lambdadef}.  This matrix can be
interpreted as the FEM matrix of a standard, convergent discretization
of the perturbed equation \eqref{eq:ucontdisturb2}.  The general
diffusion matrix $\tfatgamma(x)$ is symmetric, positive definite and
may have non-zero off-diagonal elements.  The difference between
$\fatgamma$ and $\tfatgamma$ should be as small as possible. This is a
measure of how close the jump coefficients are to modeling stochastic
diffusion which converges to isotropic diffusion with a constant
$\gamma$.


\subsection{The FEM discretization}\label{sec:FEM}

Interpreting $\tfatS$ as the standard FEM stiffness matrix to the perturbed equation \eqref{eq:ucontdisturb2} implies that
\begin{equation}
\tS_{ij} = -(\nabla\psi_i,\tfatgamma(\fatx)\nabla\psi_j)\label{eq:BackAn}
\end{equation}
for all edges $e_{ij}$. 
The sparsity pattern of $\fatS$ and $\tfatS$ is the same determined by the connectivity of the mesh.
Here $\psi_i$ and $\psi_j$ are the hat functions of linear Lagrangean finite elements with $\psi_i(\fatx_i)=1$ 
and $\psi_i(\fatx_j)=0$ when $i\ne j$.
Since the right hand side of \eqref{eq:BackAn} is a symmetric expression in $i$ and $j$ the perturbed stiffness matrix $\tfatS$ has to be symmetric. The FVM and nnFEM generate symmetric stiffness matrices. To symmetrize $\tfatS$ resulting from GFET we use its symmetric part
$(\tfatS+\tfatS^T)/2$ as $\tfatS$ in the following.
The boundary $\pOmega$ of the domain $\Omega$ is assumed 
to be polygonal and $\Omega$ is discretized such that 
\begin{equation}
\Omega=\bigcup_{T_k\in\calT}T_k,
\end{equation}
where $\calT$ is the set of all non-overlapping elements $T_k$. These elements are triangles in 2D and tetrahedra in 3D 
in the primal mesh on $\Omega$ defined by the edges $e_{ij}$, see Fig.~\ref{fig:Discretization}. The dual mesh on $\Omega$ defines the voxels $\calV_i$ in Section~\ref{sec:mesodiff}. 
With $\calT_{ij}\subset\calT$ being the set of all 
triangles in 2D or tetrahedra in 3D containing edge $e_{ij}$ we can write \eqref{eq:BackAn} as
\begin{align}\label{eq:STilde}
\tS_{ij} &= -\sum_{T_k\in\calT_{ij}}\int_{T_k}\nabla\psi_i^T\tfatgamma(\fatx)\nabla\psi_j d\fatx 
= -\sum_{T_k\in\calT_{ij}}\nabla\psi_i^T\big|_{T_k}\int_{T_k}\tfatgamma d\fatx\nabla\psi_j\big|_{T_k}\\
&= -\sum_{T_k\in\calT_{ij}}\nabla\psi_i^T\big|_{T_k}\tfatgamma_k\nabla\psi_j\big|_{T_k}|T_k|,\nonumber
\end{align}
since the gradients are constant in $T_k$. 
It is only the average $\tfatgamma_k$ 
of $\tfatgamma(x)$ on each element $T_k$ that contributes to $\tS_{ij}$.
Thus, we calculate $\tfatgamma_k$ of the following type in 2D and 3D, respectively,
\begin{equation}\label{eq:gammadef}
\tfatgamma^{2D}_{k} = 
\begin{pmatrix}
\tgamma_{k1} & \tgamma_{k3}\\
\tgamma_{k3} & \tgamma_{k2}
\end{pmatrix}
,\quad
\tfatgamma^{3D}_{k} = 
\begin{pmatrix}
\tgamma_{k1} & \tgamma_{k4} & \tgamma_{k5} \\
\tgamma_{k4} & \tgamma_{k2} & \tgamma_{k6} \\
\tgamma_{k5} & \tgamma_{k6} & \tgamma_{k3}
\end{pmatrix}.
\end{equation}
With
\begin{equation}
\nabla\psi_i^{2D}=
\begin{pmatrix}
\nabla\psi_{i1}\\
\nabla\psi_{i2}
\end{pmatrix},
\quad
\nabla\psi_i^{3D}=
\begin{pmatrix}
\nabla\psi_{i1}\\
\nabla\psi_{i2}\\
\nabla\psi_{i3}
\end{pmatrix},
\end{equation}
and the coefficients $C_{ijkl}$ and $L$ as in Table \ref{tab:coefficients}, \eqref{eq:BackAn} becomes for each edge $e_{ij}$
\begin{equation}
\tS_{ij} = \sum_{T_k\in\calT_{ij}}\sum_{l=1}^L C_{ijkl}\tgamma_{kl}.
\label{eq:constraint}
\end{equation}

\begin{table}[H]
\centering
\begin{tabular}{|l|c|c|}
\hline
& 2D & 3D\\
\hline
L & 3 & 6\\
\hline
$C_{ijk1}$ & $-\nabla\psi_{i1}\nabla\psi_{j1}$ & $-\nabla\psi_{i1}\nabla\psi_{j1}$ \\
$C_{ijk2}$ & $-\nabla\psi_{i2}\nabla\psi_{j2}$ & $-\nabla\psi_{i2}\nabla\psi_{j2}$ \\
$C_{ijk3}$ & $-(\nabla\psi_{i1}\nabla\psi_{j2}$ +$\nabla\psi_{i2}\nabla\psi_{j1})$ & $-\nabla\psi_{i3}\nabla\psi_{j3}$ \\
$C_{ijk4}$ & & $-(\nabla\psi_{i1}\nabla\psi_{j2}$ +$\nabla\psi_{i2}\nabla\psi_{j1})$ \\
$C_{ijk5}$ & & $-(\nabla\psi_{i1}\nabla\psi_{j3}$ + $\nabla\psi_{i3}\nabla\psi_{j1})$ \\
$C_{ijk6}$ & & $-(\nabla\psi_{i2}\nabla\psi_{j3}$ + $\nabla\psi_{i3}\nabla\psi_{j2})$ \\
\hline
\end{tabular}
\caption{Coefficients for in \eqref{eq:constraint}.}
\label{tab:coefficients}
\end{table}

In 2D, the integrand in \eqref{eq:BackAn} is non-zero on two triangles and on at least three tetrahedra in 3D for the edges in the interior of $\Omega$.
One can show using induction that a 2D mesh with $N$ vertices and $E_B$ edges at the boundary has $T=2N-2-E_B$ triangles and $E=3N-3-E_B$ edges. Taking into account that there are three unknowns per triangle in \eqref{eq:gammadef} 
and one equation \eqref{eq:constraint} per edge we have to solve an underdetermined system for any triangulation containing more than one triangle 
with $3N-3-2E_B$ remaining degrees of freedom.

In 3D, the system of linear equations defined by \eqref{eq:constraint} is also underdetermined if the mesh consists of more than one tetrahedron. 
Each edge in the mesh is an edge of at least one tetrahedron but there may be only one tetrahedron associated with the edge on the boundary.
Then the number of unknowns is six in \eqref{eq:gammadef} 
and the number of linear constraints \eqref{eq:constraint} is six. For each additional tetrahedron
sharing the same edge, there are six new unknowns and three new constraints. The total number of unknowns for each edge is $6T$ where $T$ is the
number of tetrahedra with a common edge and the number of linear constraints is $3T+3$. Locally, the diffusion matrix $\tfatgamma$
is underdetermined with $3T-3$ degrees of freedom. 

Consequently, the diffusion $\tfatgamma$ satisfying \eqref{eq:constraint} is not unique,
but for all possible $\tfatgamma$ the error analysis in Section~\ref{sec:ErrorEstimate} holds. We obtain the sharpest bounds on $\|u-\tilde{u}\|_{L^2}$ and $\|\exit-\tilde{\exit}\|_{L^2}$ by finding $\tfatgamma$ satisfying \eqref{eq:constraint} and minimizing the difference $\|\fatgamma-\tfatgamma\|_*$ in the equations. An alternative would be to replace $\|\cdot\|_2$ in \eqref{eq:GlobalError} by the Frobenius norm $\|\cdot\|_F$.
Since 
\begin{equation}\label{eq:normineq}
  \|\fatA\|_2\le\|\fatA\|_F\le\sqrt{d}\|\fatA\|_2
\end{equation}
for a matrix $\fatA$ \cite{GVL}, the bound in Section~\ref{sec:ErrorEstimate} is sharper if the minimization is made in the $\|\cdot\|_2$ norm.
In the following, we propose a 
global and a local optimization procedure to find these minimizers $\tfatgamma$.

\subsection{Global optimization}\label{sec:gloanalysis}
The diffusion matrix $\tfatgamma$ closest to the original diffusion $\fatgamma$ with constant coefficient $\gamma$ is found by minimizing the distance between $\tfatgamma(\fatx)$ and $\fatgamma$ under the constraints in \eqref{eq:constraint}. The stiffness matrix $\tfatS$ is given by \eqref{eq:lambdadef} and one of the methods in Section~\ref{sec:mesodiff}. As only the average $\tfatgamma_k$ of $\tfatgamma(\fatx)$ appears in the FEM approximation on each triangle $T_k$, see \eqref{eq:STilde}, the norm of the difference in diffusion in \eqref{eq:GlobalError} reduces to the weighted sum of the differences $\|\tfatgamma_k-\fatgamma\|_2^2$ as a measure of the distance resulting in the following optimization problem
\begin{eqnarray}\label{eq:globopt}
\min_{\tfatgamma_k} &\sum_{T_k\in\calT}|T_k|\|\tfatgamma_k-\fatgamma\|_2^2 &\label{eq:objfunc}\\
&\sum_{T_k\in\calT_{ij}}\sum_{l=1}^L C_{ijkl}\tgamma_{kl} &= \tS_{ij},\quad\forall e_{ij}.\label{eq:linconstr}
\end{eqnarray}
This is a nonlinear programming problem with $3T$ variables in 2D and $6T$ variables in 3D and $E$ linear constraints.

The difference $\|\tfatgamma_k-\fatgamma\|_2^2$ is a convex function in the unknowns in $\tfatgamma_k$. Hence, 
the objective function in \eqref{eq:objfunc} is a convex function too. Since also the constraint set in \eqref{eq:linconstr} is convex, the
local solution to \eqref{eq:objfunc} and \eqref{eq:linconstr} is the unique global optimum. If $S_{ij}\ge 0$ for all $i,j$ from the FEM discretization
with diffusion constant $\gamma$, then $\tS_{ij}=S_{ij}$ for all $i,j$ and the solution to \eqref{eq:globopt} is $\tfatgamma_k=\fatgamma$ for all $T_k$.

The mean value matrix $\tfatgamma_k$ defines two (three) main axes in 2D (3D) on $T_k$. Let the columns of $\fatV$ be the eigenvectors
$\fatv_j$ of $\tfatgamma_k$ with eigenvalues $\lambda_j$. After a coordinate transformation from $\fatx$ to $\faty$ with 
$\fatx=\fatV\faty$ the diffusion term is
\begin{equation}\label{eq:coortrans}
   \nabla_{\fatx}\cdot(\tfatgamma\nabla_{\fatx} u)=\sum_j \lambda_j\frac{\partial^2 u}{\partial v_j^2}. 
\end{equation} 
The eigenvectors define the main axes of the diffusion and the diffusion speed along those axes is given by the eigenvalues of $\tfatgamma$.
Since 
\begin{equation}\label{eq:gamdif2norm}
    \|\tfatgamma_k-\fatgamma\|_2=\max_j|\lambda_j-\gamma|,
\end{equation} 
the $\ell_2$ norm in \eqref{eq:objfunc} 
measures the maximum deviation in speed of the diffusion in $\tfatgamma$ compared to $\gamma$ weighted by the size of $T_k$. 
In the Frobenius norm 
\begin{equation}\label{eq:gamdifFnorm}
    \|\tfatgamma_k-\fatgamma\|_F=\left(\sum_{j=1}^d (\lambda_j-\gamma)^2\right)^{1/2},
\end{equation} 
and the norm is equal to the $\ell_2$ norm of the difference in diffusion speed in all directions.
The objective function in \eqref{eq:objfunc}
is continuous in $\tfatgamma$ but it is not continuously differentiable everywhere.

\subsection{Local optimization}\label{sec:locanalysis}

The optimization problem in the previous section may be computationally expensive but it is simplified if 
we approach the solution of \eqref{eq:linconstr} by local optimization. Let $\calE_{ij}$ be defined by
\begin{equation}\label{eq:Eijdef}
  \calE_{ij}=\{e_{mn}: \; e_{mn}\;{\rm is}\;{\rm an}\;{\rm edge}\;{\rm of}\;{\rm any}\; T_k\in\calT_{ij}\}
\end{equation}
The adjacent $\tfatgamma_k$ in $T_k$ in 
$\calT_{ij}$ for each edge $e_{ij}$ are optimized, while keeping $\tS_{ij}$ constant on the other 
edges in $\calE_{ij}$. 
Update $\tfatgamma_k$ with the most recently computed diffusion matrix. Then iterate over all edges once.
Still, the underdetermined system \eqref{eq:linconstr} will be satisfied
but with a different $\tfatgamma^L$ compared to $\tfatgamma^G$ solving \eqref{eq:objfunc}. The algorithm is as follows
\begin{algorithm}[H]
\caption{Local Optimization I}
\label{alg:Local}
\begin{algorithmic}[1]
\State $\tfatgamma_k=\fatgamma,\quad\forall T_k\in\calT$
\ForAll {$e_{ij}$}
\State Solve
\begin{equation*}
\begin{array}{lll}
\min_{\tfatgamma_k^{new}} & \sum_{T_k\in\calT_{ij}}|T_k|\|\tfatgamma_k^{new}-\fatgamma\|_2^2\\
&\sum_{T_k\in\calT_{ij}}\sum_{l=1}^L C_{ijkl}\tgamma_{kl}^{new} &= \tS_{ij}\\
&\sum_{l=1}^L C_{mnkl}\tgamma_{kl}^{new} &= \sum_{l=1}^L C_{mnkl}\tgamma_{kl}\quad\\
&&\forall e_{mn}\in\calE_{ij}\setminus e_{ij},\; \forall T_k\in\calT_{ij}
\end{array}
\end{equation*}
\State $\tfatgamma_k=\tfatgamma_k^{new},\, T_k\in\calT_{ij}$
\EndFor
\end{algorithmic}
\end{algorithm}

The diffusion $\tfatgamma_k$ changes successively only on the elements adjacent to $e_{ij}$ 
(two triangles in 2D and at least three tetrahedra in 3D in the interior) 
in each iterative step. 
 At each inner edge in 2D, there are six variables and five constraints.
As remarked in Section~\ref{sec:FEM} above, the number of variables in 3D in Algorithm~\ref{alg:Local} 
is $6T$ and the number of constraints is $3T+3$, where
$T$ is the number of tetrahedra sharing the common edge $e_{ij}$. 
For a boundary edge, the number of unknowns equals the number of constraints in 2D and one has to solve only
the linear system in \eqref{eq:constraint}.

Also here we have that if $\tS_{ij}=S_{ij}\ge 0$ for all $i,j$ then 
the solution is $\tfatgamma_k=\fatgamma$. The order in which the edges are traversed matters for the result by the algorithm but when all edges have been
visited then \eqref{eq:linconstr} is satisfied. In the numerical experiments in Section~\ref{sec:numres}, the order
is random but other choices are possible.

The global $\tfatgamma_k^G$ from \eqref{eq:objfunc} and the local $\tfatgamma_k^L$ from Algorithm~\ref{alg:Local} 
fulfill
\begin{equation}\label{eq:glopt}
\eta^G_2 = \sqrt{\frac{1}{|\Omega|}\sum_{T_k\in\calT}|T_k|\|\tfatgamma_k^G-\fatgamma\|_2^2} \leq \sqrt{\frac{1}{|\Omega|}\sum_{T_k\in\calT}|T_k|\|\tfatgamma_k^L-\fatgamma\|_2^2} = \eta^L_2,
\end{equation}
since $\tfatgamma_k^G$ is the global minimum solution.

Neither the global nor the local procedure to determine $\tfatgamma$ guarantee its positive definiteness when the solution is computed with
an optimization algorithm for a nonlinear objective function with linear constraints. Extra nonlinear constraints can be added to enforce positive definiteness. That leads to slow algorithms or sometimes very large backward errors $\|\fatgamma-\tfatgamma_k\|_2$ in the numerical experiments in Section~\ref{sec:numres}. An alternative would be to apply a computationally more expensive semi-definite programming algorithm \cite{BoydVanden04, VandenBoyd96} to the problem. In Section~\ref{sec:numres}, we first compute $\tfatgamma$ without constraints for positive definiteness and then check the solution for positive definiteness. The nonlinear programming algorithm finds positive definite $\tfatgamma_k$ for
all elements in most cases.

The diffusion $\tfatgamma$ is computed for a given mesh of finite mesh size $h$. What happens with $\tfatgamma$ when the mesh is refined depends on the 
mesh generator. If all $S_{ij}$ become non-negative as $h\rightarrow 0$, then $\tfatgamma\rightarrow \fatgamma$. Otherwise, there will be a difference 
$\|\fatgamma-\tfatgamma\|_*$ of $\ordo{1}$ as $h$ vanishes.

The backward analysis is extended in the next section to the design of the stiffness matrix $\tfatS$ such that the backward error $\|\fatgamma-\tfatgamma\|_*$ is minimized. 


\section{Design}\label{sec:design}

In the previous section, we described how to analyze existing methods for creating positive jump coefficients by backwards analysis. In this section we determine a new discretization using FEM by minimizing the backward error.
We devise non-negative jump coefficients such that the perturbed diffusion $\tfatgamma$ is as close as possible to the original diffusion with a constant $\gamma$.
The connectivity of the network of edges is the same as in Section~\ref{sec:FEM} but $\tS_{ij}$ is free
to vary. Molecules in the stochastic setting are allowed to jump only to the neighboring voxels but the rate is a free variable to be optimized such that the distribution of molecules converges to the diffusion equation \eqref{eq:ucontdisturb2} in the limit of large molecules numbers.

\subsection{Global optimization}\label{sec:glodesign}

The diffusion $\tfatgamma_k$ in each triangle or tetrahedron is determined such that
\begin{eqnarray}\label{eq:globopt2}
\min_{\tfatgamma_k} &\sum_{T_k\in\calT}|T_k|\|\tfatgamma_k-\fatgamma\|_2^2 &\label{eq:objfunc2}\\
&\sum_{T_k\in\calT_{ij}}\sum_{l=1}^L C_{ijkl}\tgamma_{kl} &\ge 0,\quad\forall e_{ij}.\label{eq:linconstr2}
\end{eqnarray}
The equality constraints in \eqref{eq:linconstr} are replaced by the inequalities in \eqref{eq:linconstr2}. 
The new jump coefficients $\lambda_{ji}$ are computed by the optimal $\tfatgamma$ and $\tS_{ij}$  
\begin{equation}\label{eq:newS}
   \tS_{ij}=\sum_{T_k\in\calT_{ij}}\sum_{l=1}^L C_{ijkl}\tgamma_{kl},\; i\ne j,
\end{equation}
inserted into \eqref{eq:lambdadef}. 
The stiffness matrix is thus obtained from the FEM discretization of the diffusion term in \eqref{eq:ucontdisturb1} with linear Lagrangean elements 
and diffusion matrix $\tfatgamma_k$ on $T_k$. 

\subsection{Local optimization}\label{sec:locdesign}

The local optimization algorithm in Section~\ref{sec:locanalysis} to analyze given jump coefficients is modified in the same way to
generate new coefficients instead. For each edge, the adjacent diffusion matrices are computed such that they are close to
$\fatgamma$ and the non-negativity constraint is satisfied for the edges. 
Instead of keeping the contribution to the other edges constant we let it vary constrained by non-negativity. 
If $e_{mn}$ is an edge in $\calE_{ij}\setminus e_{ij}$, then we allow $\tfatgamma$ to be such that
\begin{equation}
\sum_{T_k\in\calT_{mn}}\sum_{l=1}^L C_{mnkl}\tgamma_{kl} =\tS_{mn}\ge 0.
\label{eq:PosOuterEdge}
\end{equation}
In each local optimization step, $\tfatgamma_k$ in the elements $T_k$ adjacent to edge $e_{ij}$ are modified while keeping $\tS_{mn}$
in other edges in $\calE_{ij}$ non-negative. Splitting the sum in \eqref{eq:PosOuterEdge} into two parts we have
\begin{eqnarray}
\sum_{T_k\in\calT_{ij}\cap\calT_{mn}}\sum_{l=1}^L C_{mnkl}\tgamma_{kl} &\ge & -\sum_{T_k\in\calT_{mn}\setminus(\calT_{ij}\cap\calT_{mn})}\sum_{l=1}^L C_{mnkl}\tgamma_{kl}
\label{eq:diffupda}\\
&=& \sum_{T_k\in\calT_{ij}\cap\calT_{mn}}\sum_{l=1}^L C_{mnkl}\tgamma_{kl}-\tS_{mn}.
\label{eq:diffupd}
\end{eqnarray}
The diffusion matrix on the left hand side of \eqref{eq:diffupda} is updated given the diffusion matrix in the right hand side in \eqref{eq:diffupd}.
This is repeated successively for all edges in the following algorithm.
\begin{algorithm}[H]
\caption{Local Optimization II}
\label{alg:Local2}
\begin{algorithmic}[1]
\State $\tfatgamma_k=\fatgamma, \quad\forall T_k\in\calT$ 
\ForAll {$e_{ij}$} 
\State $\tS_{ij}=\sum_{T_k\in\calT_{ij}}\sum_{l=1}^L C_{ijkl}\tgamma_{kl}$
\EndFor
\ForAll {$e_{ij}$}
\State Solve
\begin{equation*}
\begin{array}{lll}
\min_{\tfatgamma_k^{new}} & \sum_{T_k\in\calT_{ij}}|T_k|\|\tfatgamma_k^{new}-\fatgamma\|_2^2\\
&\sum_{T_k\in\calT_{ij}}\sum_{l=1}^L C_{ijkl}\tgamma_{kl}^{new} &\ge 0\\
&\sum_{T_k\in\calT_{ij}\cap\calT_{mn}}\sum_{l=1}^L C_{mnkl}\tgamma_{kl}^{new} &\ge \sum_{T_k\in\calT_{ij}\cap\calT_{mn}}\sum_{l=1}^L C_{mnkl}\tgamma_{kl}-\tS_{mn} \quad\\
&&\forall e_{mn}\in\calE_{ij}\setminus e_{ij},\; \forall T_k\in\calT_{ij}
\end{array}
\end{equation*}
\State $\tfatgamma_k=\tfatgamma_k^{new},\, T_k\in\calT_{ij}$
\State $\tS_{ij}=\sum_{T_k\in\calT_{ij}}\sum_{l=1}^L C_{ijkl}\tgamma_{kl}$
\EndFor
\end{algorithmic}
\end{algorithm}

The number of optimization problems to be solved in Algorithm~\ref{alg:Local2} is the number of edges $E$ which is bounded by a constant times
the number of vertices $N$ in a mesh.
The size of each optimization problem is independent of $E$ and $N$. Hence, the computational work is proportional to $N$ in the local
optimization and of the same computational complexity as the matrix assembly of $\fatS$.
The edges are traversed in a random order in the experiments in Section~\ref{sec:numres}. When all edges have been visited once, 
$\tS_{ij}$ satisfies \eqref{eq:newS} and $\tS_{ij}\ge 0$ and the new
$\lambda_{ji}$ is computed using \eqref{eq:lambdadef}.

\subsection{Practical implementation}
The local minimization problem only contains the adjacent triangles or tetrahedra and is hence faster to compute but $\eta^L_2>\eta^G_2$, see \eqref{eq:glopt}.
Instead of running the local Algorithms~\ref{alg:Local} and~\ref{alg:Local2}  only once, we can repeat them iteratively with the results $\tfatgamma_k$ of the previous iteration as the initial guess for the next minimization. 
Then $\eta^L_2$ will approach $\eta^G_2$.

A possibility to speed up the computation is to replace the $\ell_2$ norm of the error by the Frobenius norm
\begin{equation}
\eta_F = \sqrt{\frac{1}{|\Omega|}\sum_{T_k\in\calT}|T_k|\|\tilde{\fatgamma}_k-\fatgamma\|_F^2}.
\label{eq:Frobenius}
\end{equation}
The global non-linear minimization problem \eqref{eq:globopt} then simplifies to the quadratic programming problem
\begin{eqnarray}
\min_{\bar{\gamma}} &\bar{\fatgamma}^T\fatH\bar{\fatgamma}-2\fatf^T\bar{\fatgamma} \label{eq:globoptFrobenius}\\
&\sum_{T_k\in\calT_{ij}}\sum_{l=1}^L C_{ijkl}\tgamma_{kl} &\ge 0,\quad\forall e_{ij},\label{eq:globoptFrobenius2}
\end{eqnarray}
where $\fatH$ is a diagonal matrix with positive elements on the diagonal and $\bar{\fatgamma}$ is a vector with $\tgamma_{kl}$ as
components.
By the relation between the $\ell_2$ and Frobenius norms \eqref{eq:normineq}, the resulting $\eta_F$ only yields an upper bound on the global minimum $\eta^G_2$.
In the local optimizations in Algorithms~\ref{alg:Local} and \ref{alg:Local2}, $\|\cdot\|_2^2$ is then substituted by $\|\cdot\|_F^2$.

We can further reduce the computational complexity by rewriting the high dimensional minimization problem \eqref{eq:globoptFrobenius} 
as the smaller dual problem.
\begin{equation}\label{eq:globoptDual}
\min_{\fatmu\geq 0} \fatmu^T\tilde{\fatH}\fatmu+2\tilde{\fatf}^T\fatmu,
\end{equation}
where $\fatmu\geq 0$ is equivalent to $\mu_i\geq 0, \forall i,$ and
\begin{equation}\label{eq:Cconstr}
\tilde{\fatH} = \fatC \fatH^{-1}\fatC^T,\quad\quad
\tilde{\fatf} = \fatC\fatH^{-1}\fatf,\quad\quad
\fatgamma = -\fatH^{-1}(\fatC^T\fatmu-\fatf).
\end{equation}
In \eqref{eq:Cconstr}, $\fatC$ is such that \eqref{eq:globoptFrobenius2} is replaced by $\fatC\bar{\fatgamma}\ge 0$.
The primal problem of dimension $3T$ is hence reduced by approximately a factor two to the dual problem of dimension $E$ in 2D. In 3D the dual problem is more than a factor 4.5 smaller than the primal problem in numerical experiments in Section~\ref{sec:numres}.
The interior point algorithm is well suited for the quadratic programming problems \eqref{eq:globoptFrobenius} and \eqref{eq:globoptDual}, see e.g. \cite{BoydVanden04, KiKohLusBoGo07}.

\subsection{Alternatives to determine a non-negative $\tfatS$}\label{sec:alt}

Another two possibilities are investigated to calculate a $\tfatS$ with only non-negative off-diagonal entries. From the set of discrete equations \eqref{eq:udisc} it appears that a smaller difference between the discretization matrices $\fatD$ and $\tfatD$ leads to a smaller error in the solution $\|\fatu-\tilde{\fatu}\|_{L^2}$. That suggests to find a $\tfatS$ with the same sparsity pattern as $\fatS$ but with only non-negative entries such that $\|\fatD-\tfatD\|_2$ is minimized.

A second alternative to guarantee non-negative jump coefficients is adding artificial viscosity to the system. The same viscosity is added patchwise in all elements with a common vertex. If edge $e_{ij}$ corresponds to a negative entry, then enough viscosity to eliminate the negative entry $S_{ij}$ is added to all edges originating from $\fatx_i$ and $\fatx_j$ as in the graph Laplacian. The symmetry of the original matrix $\fatS$ is preserved by adding $|S_{ij}|/2$ to the nodes around $\fatx_i$ and $\fatx_j$ in the following way
\begin{align*}
&\tS_{ik} = S_{ik}+|S_{ij}|/2,
&\tS_{ki} = S_{ki}+|S_{ij}|/2,\quad\quad\forall \fatx_k \text{ connected to } \fatx_i \text{ by } e_{ki},\\
&\tS_{jk} = S_{jk}+|S_{ij}|/2,
&\tS_{kj} = S_{kj}+|S_{ij}|/2,\quad\quad\forall \fatx_k \text{ connected to } \fatx_j \text{ by } e_{kj}.
\end{align*}
This is a generalization of the nnFEM approach where a sufficient amount of viscosity is added only to the negative edge.
This type of artificial viscocity is introduced in \cite{GuNa14} to prove that the maximum principle is satisfied for a conservation law.


\section{Numerical experiments}\label{sec:numres}

In this section, we determine numerically the local $\eta^L$ and global $\eta^G$ 
backward errors in \eqref{eq:glopt} for the different methods generating non-negative coefficients as described in Sections~\ref{sec:mesodiff} and~\ref{sec:design} with a diffusion coefficient $\gamma=1$. By the analysis in Section~\ref{sec:ErrorEstimate}, 
the backward error bounds the forward error of the mean values in the spatial distribution of the copy numbers of the molecules $\|u-\tilde{u}\|_{L^2}$ \eqref{eq:FEU} and the exit 
times $\|\exit-\tilde{\exit}\|_{L^2}$ \eqref{eq:FEE}. 
All computations are done in Matlab using its optimization routines. The meshes are generated by COMSOL Multiphysics and the FEM matrices are assembled by the same software. 


\subsection{Diffusion in 2D}
The square $[-0.5,0.5]\times[-0.5,0.5]$ is discretized into 227 nodes, see Fig.~\ref{fig:Mesh}. As mentioned in Section~\ref{sec:mesodiff}, mesh generators usually produce good quality meshes in 2D and the mesh in Fig.~\ref{fig:Mesh} is intentionally perturbed to obtain 47 edges with negative jump coefficients marked by red in Fig.~\ref{fig:Mesh}.

The requirement to obtain a non-negative discretization poses other constraints on the mesh than what is necessary for a FEM solution of high accuracy. Examples of quality measures $\calq$ related to errors in the finite element solution of Poisson's equation are found in \cite{Shew}. In 2D, let $h_1, h_2, h_3,$ be the lengths of the edges of a triangle of area $A$ with the angle $\varphi_3$ opposing the edge of maximum length $h_3$. A bound on the error in the gradient between $f$ and the approximating $f_h$ in the triangle is in \cite{Shew}
\begin{equation}\label{eq:errest}
  \displaystyle{\|\nabla f-\nabla f_h\|_\infty \le c_f \frac{3 h_1 h_2 h_3}{2A}=c_f\frac{3h_3}{2\sin(\varphi_3)}=c_f\frac{1}{\calq}},
\end{equation}
where $c_f$ is a bound on the second derivatives of $f$. The measure $\calq$ is positive and should be as large as possible. Suppose that two triangles with the same edge lengths have the edge $e_{ij}$ of length $h_3$ in common. Then $S_{ij}$ in \eqref{eq:S2D} is negative when $\varphi_3>\pi/2$ while the estimate in \eqref{eq:errest} is as small as possible when $\varphi_3$ is in the neighborhood of $\pi/2$. 
On the other hand, the accuracy is poor if $h_3$ is large and all angles are less than $\pi/2$. Then $\calq$ is small but the jump coefficients are positive.
A large $h_3$ will of course also affect the spatial resolution of the stochastic simulations but the diffusion propensities are well defined.
\begin{figure}[H]
\centering
\includegraphics[width=0.6\textwidth]{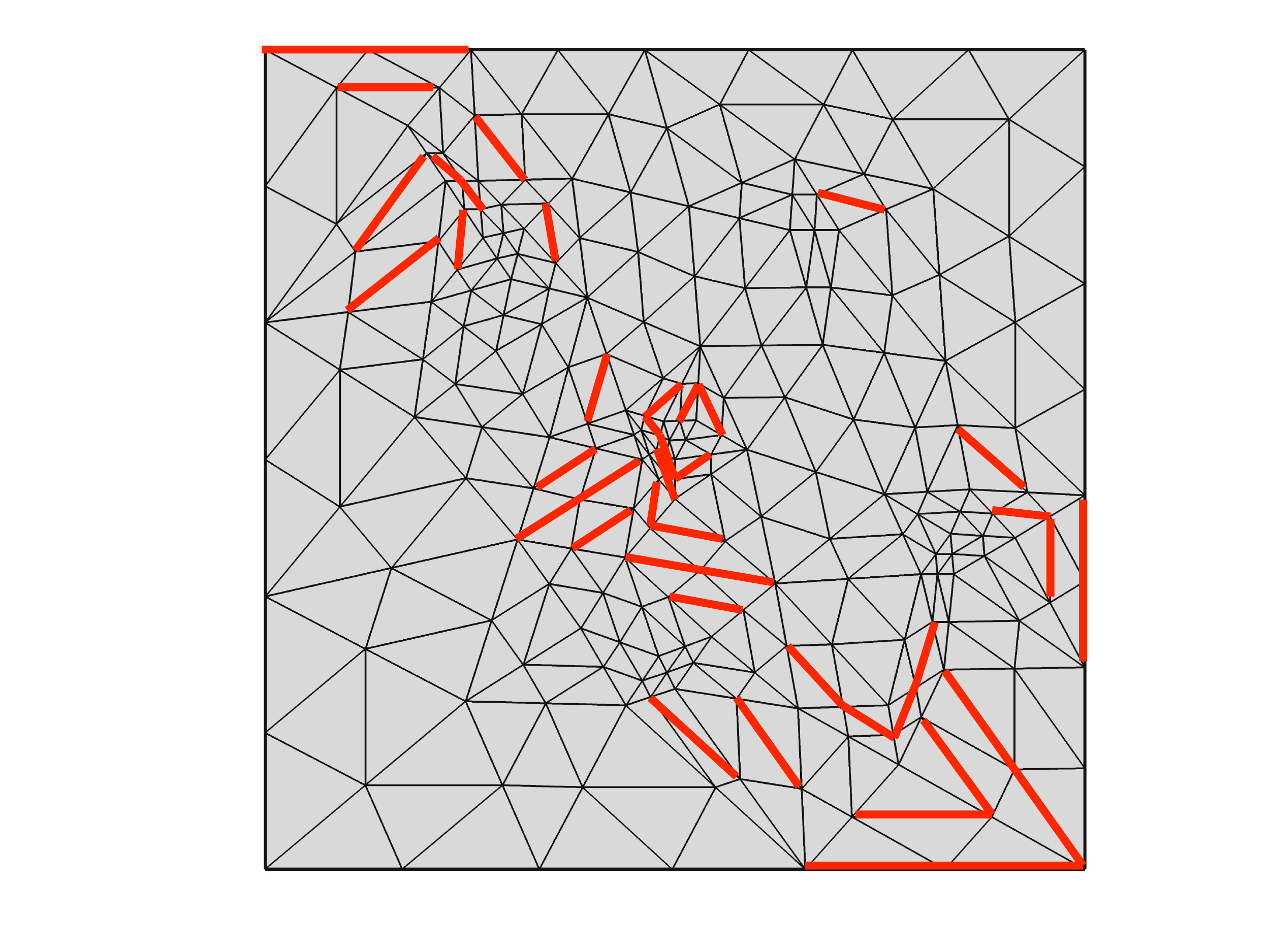} 
\caption{ The mesh in 2D. Negative edges are shown in  red. }
\label{fig:Mesh}
\end{figure}

\subsubsection{Backward analysis}
\begin{figure}[H]
\centering
\begin{tabular}{l c c}
&
\includegraphics[width=0.12\textwidth]{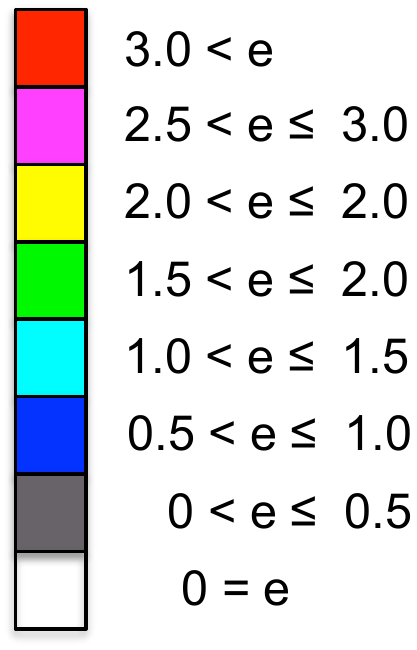} &
\includegraphics[width=0.12\textwidth]{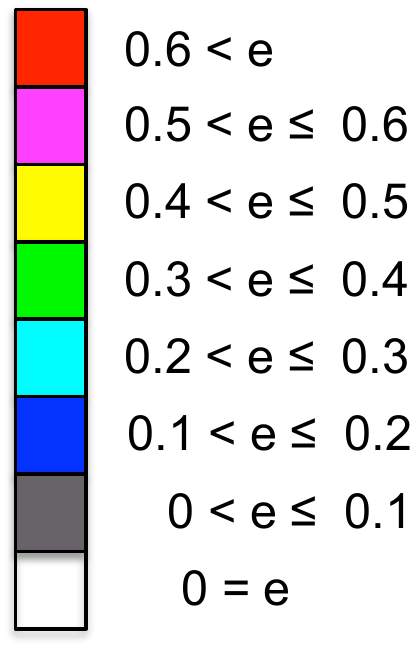}
\\
& Local & Global \\
FVM	&
\includegraphics[width=0.35\textwidth]{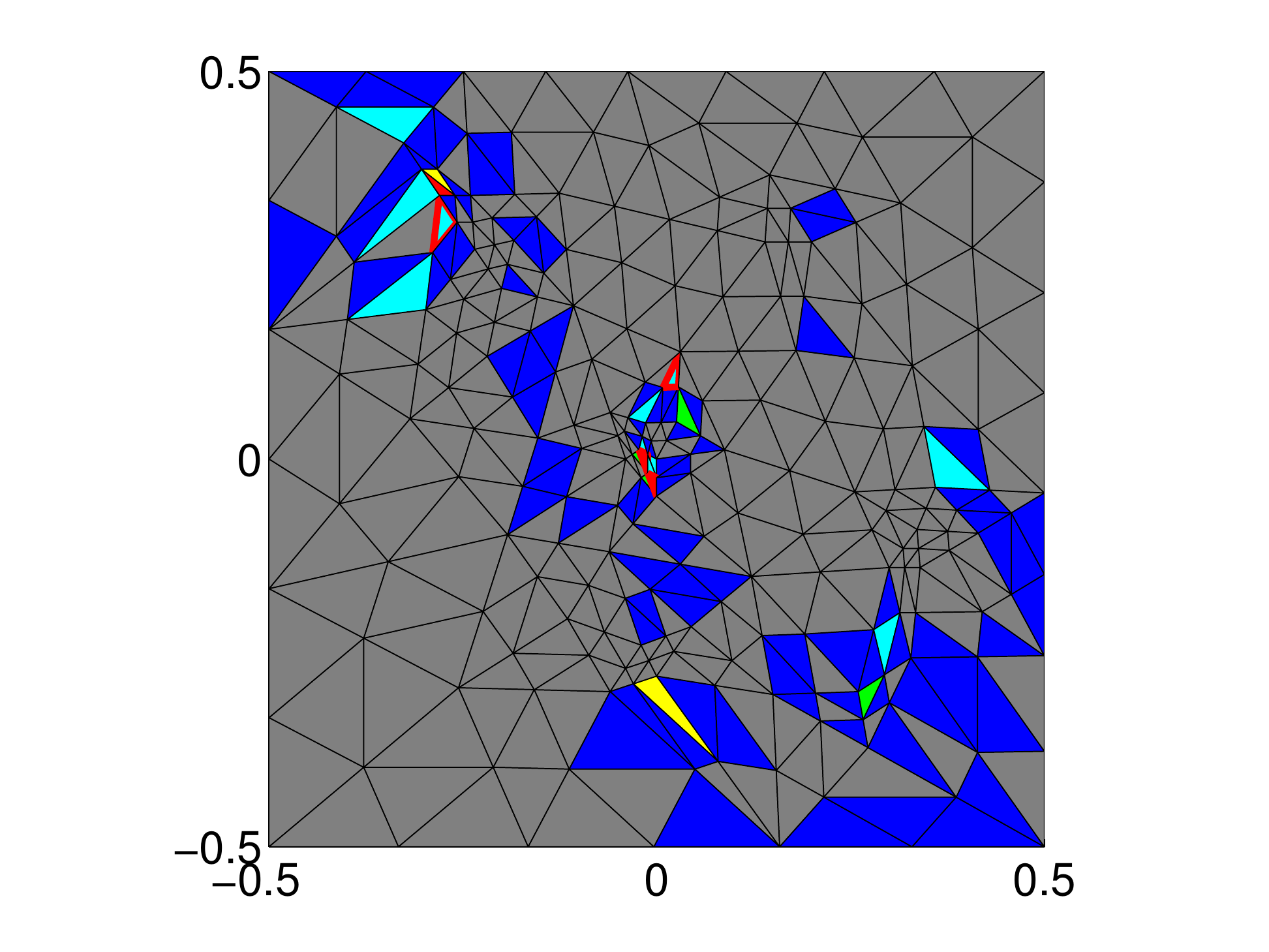} &
\includegraphics[width=0.35\textwidth]{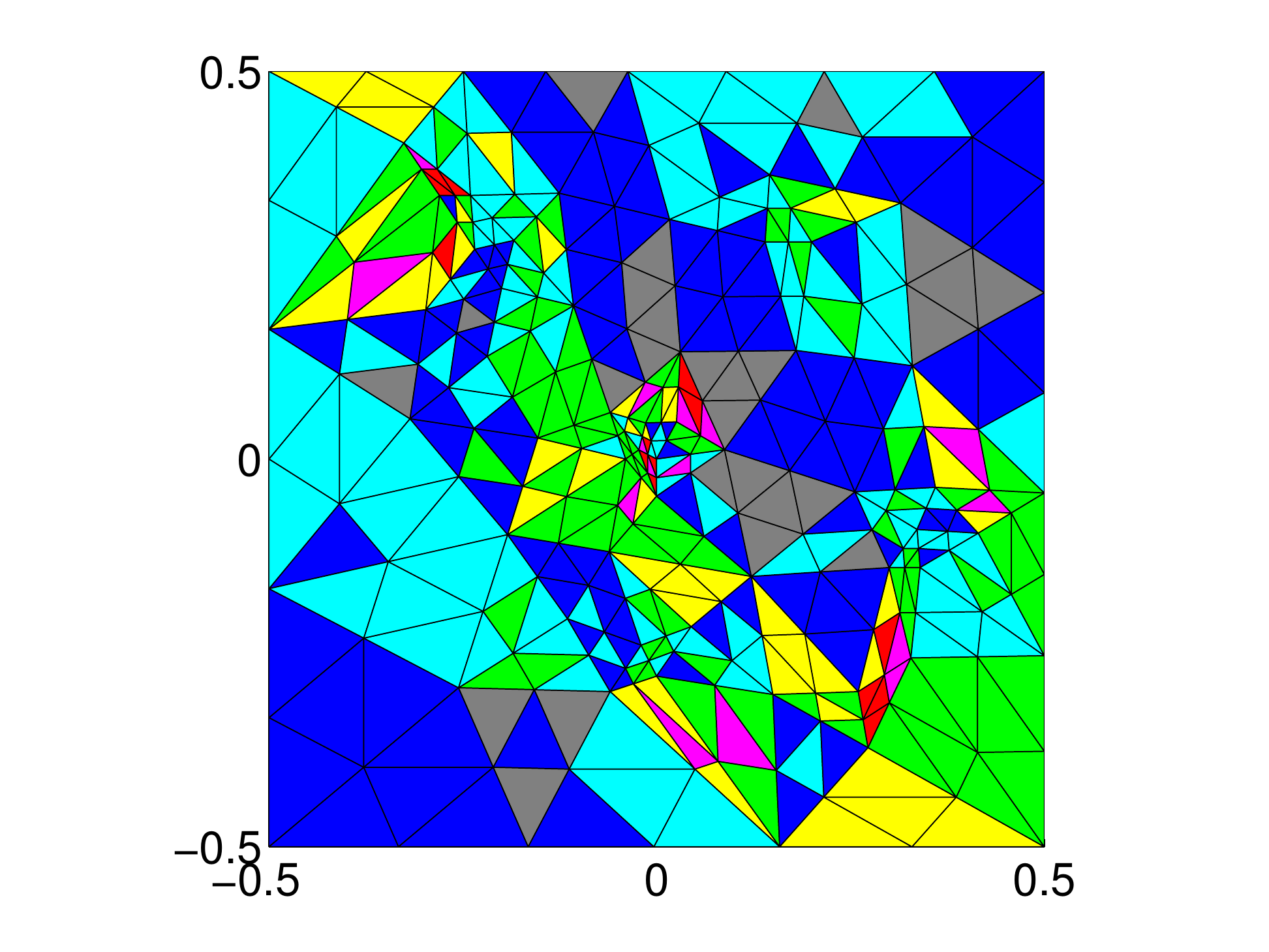}\\
GFET	&
\includegraphics[width=0.35\textwidth]{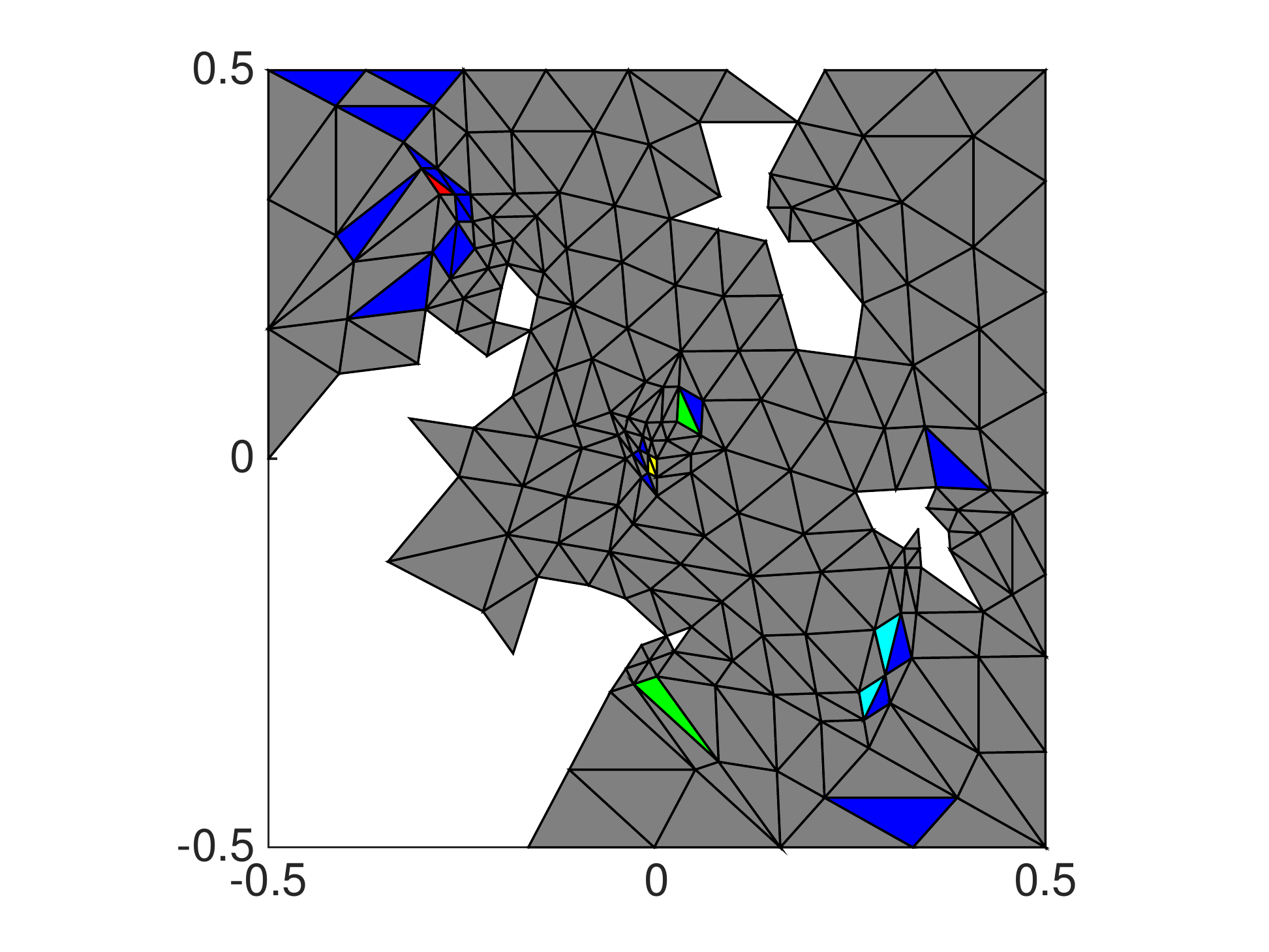} &
\includegraphics[width=0.35\textwidth]{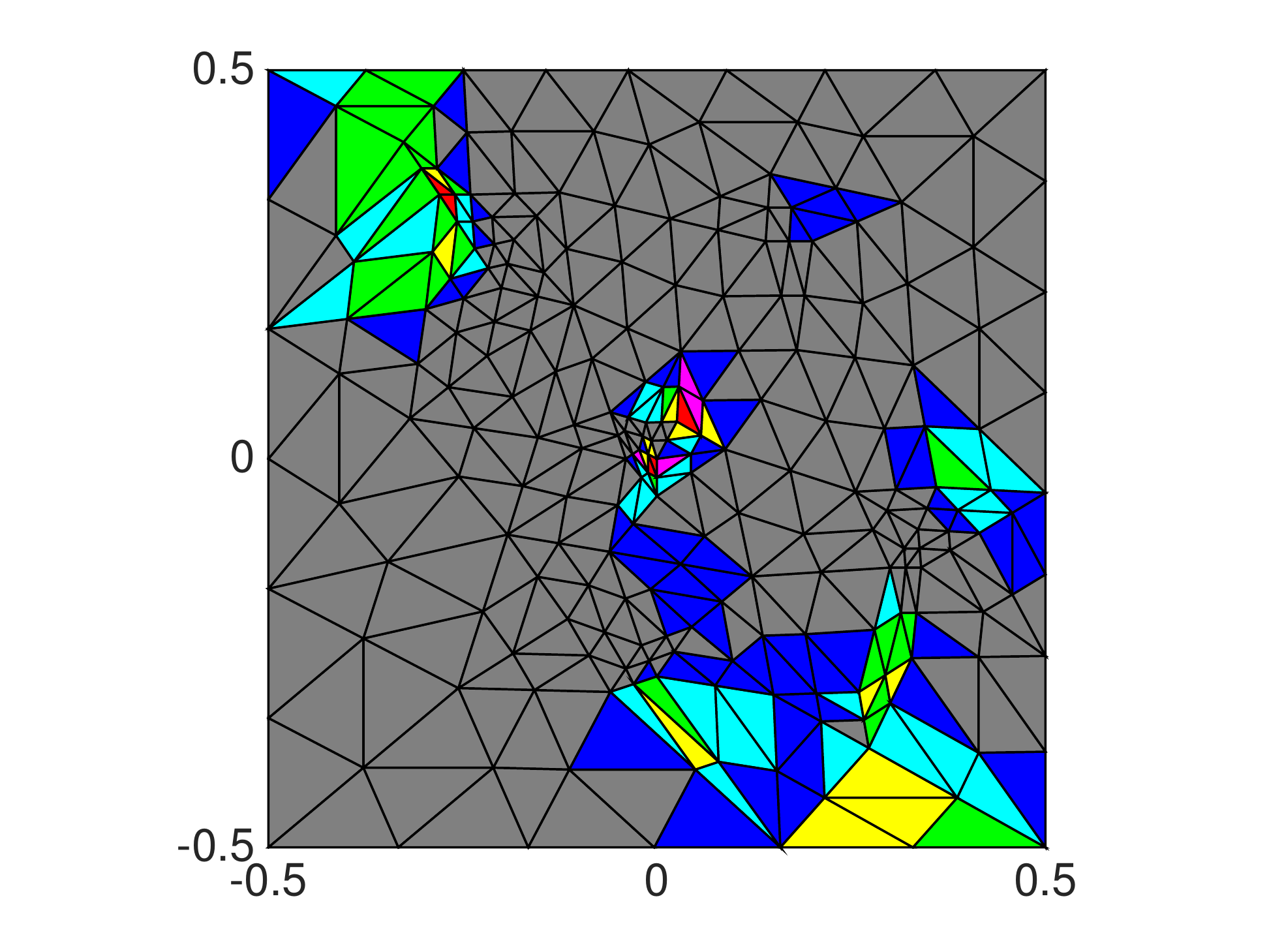} \\
nnFEM	&
\includegraphics[width=0.35\textwidth]{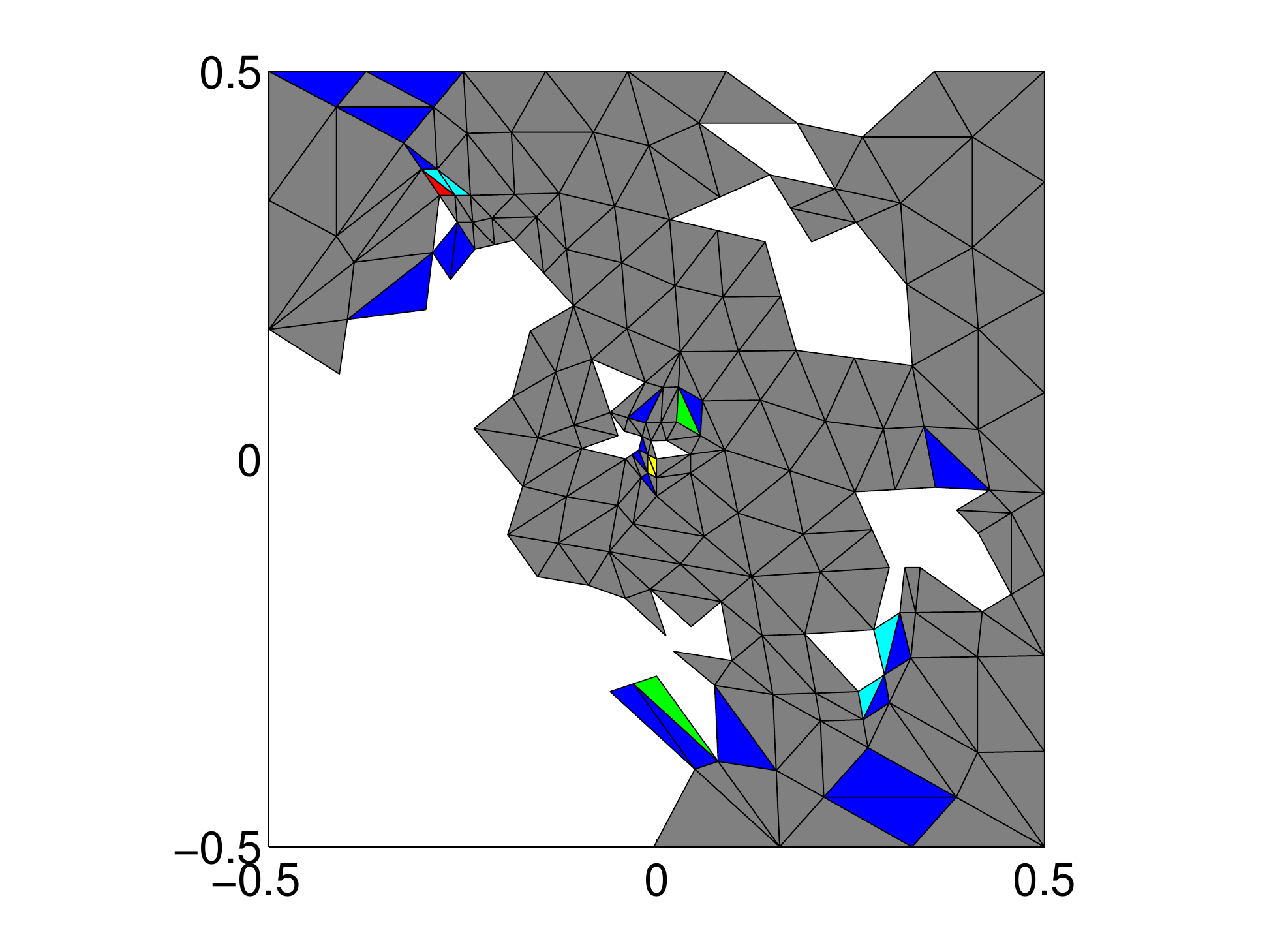} &
\includegraphics[width=0.35\textwidth]{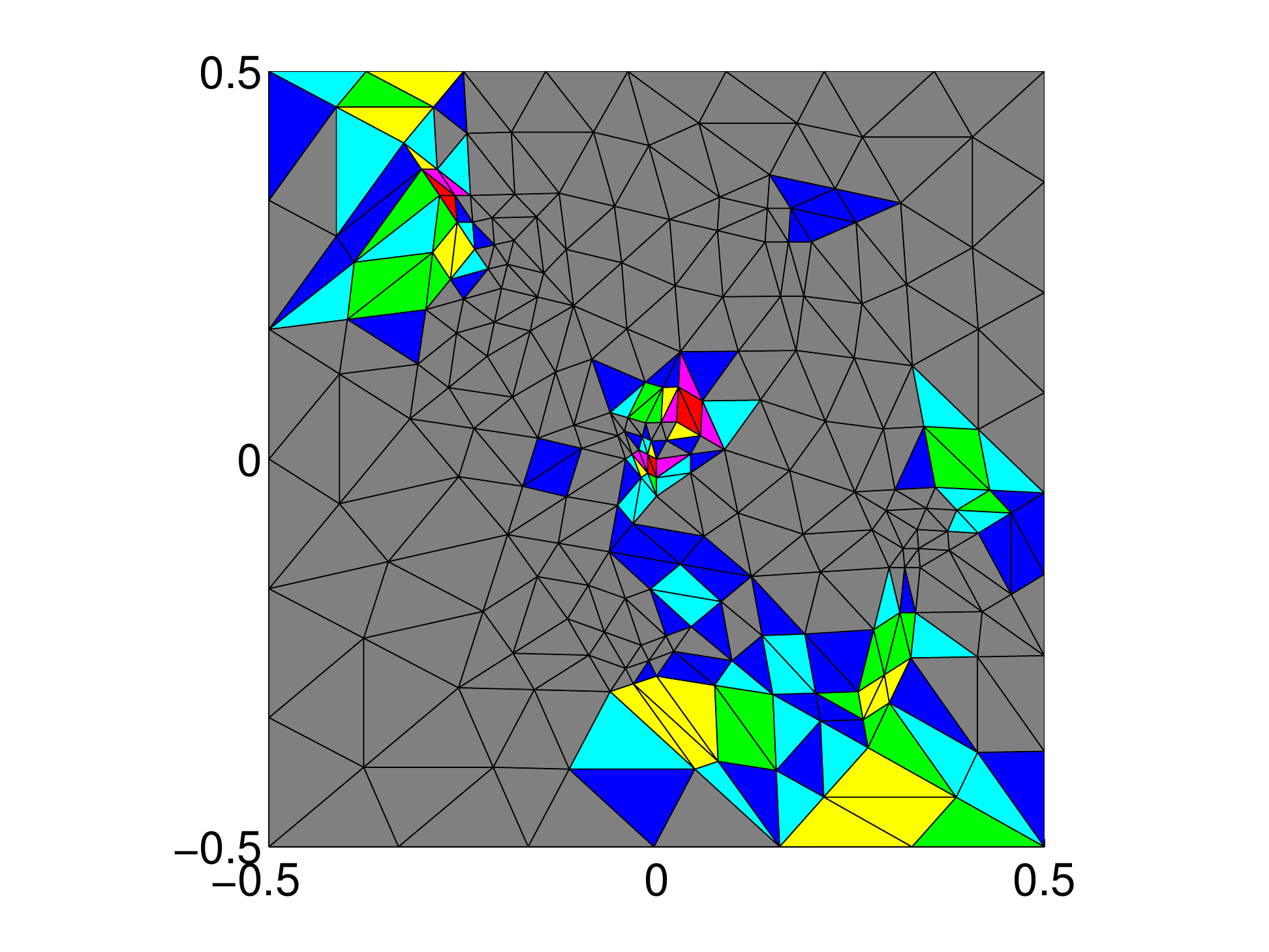} 
\\
MBE	&
\includegraphics[width=0.35\textwidth]{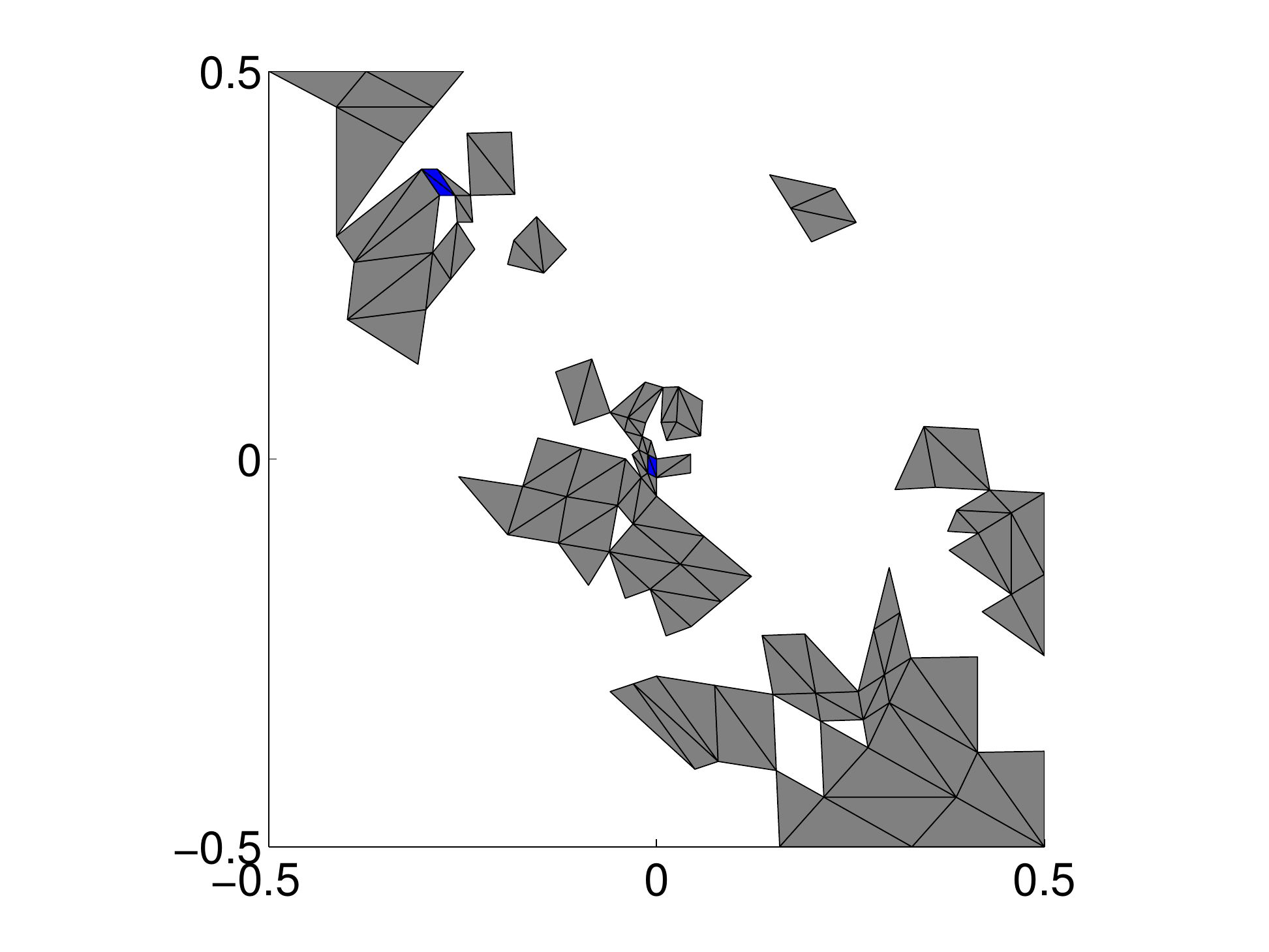} &
\includegraphics[width=0.35\textwidth]{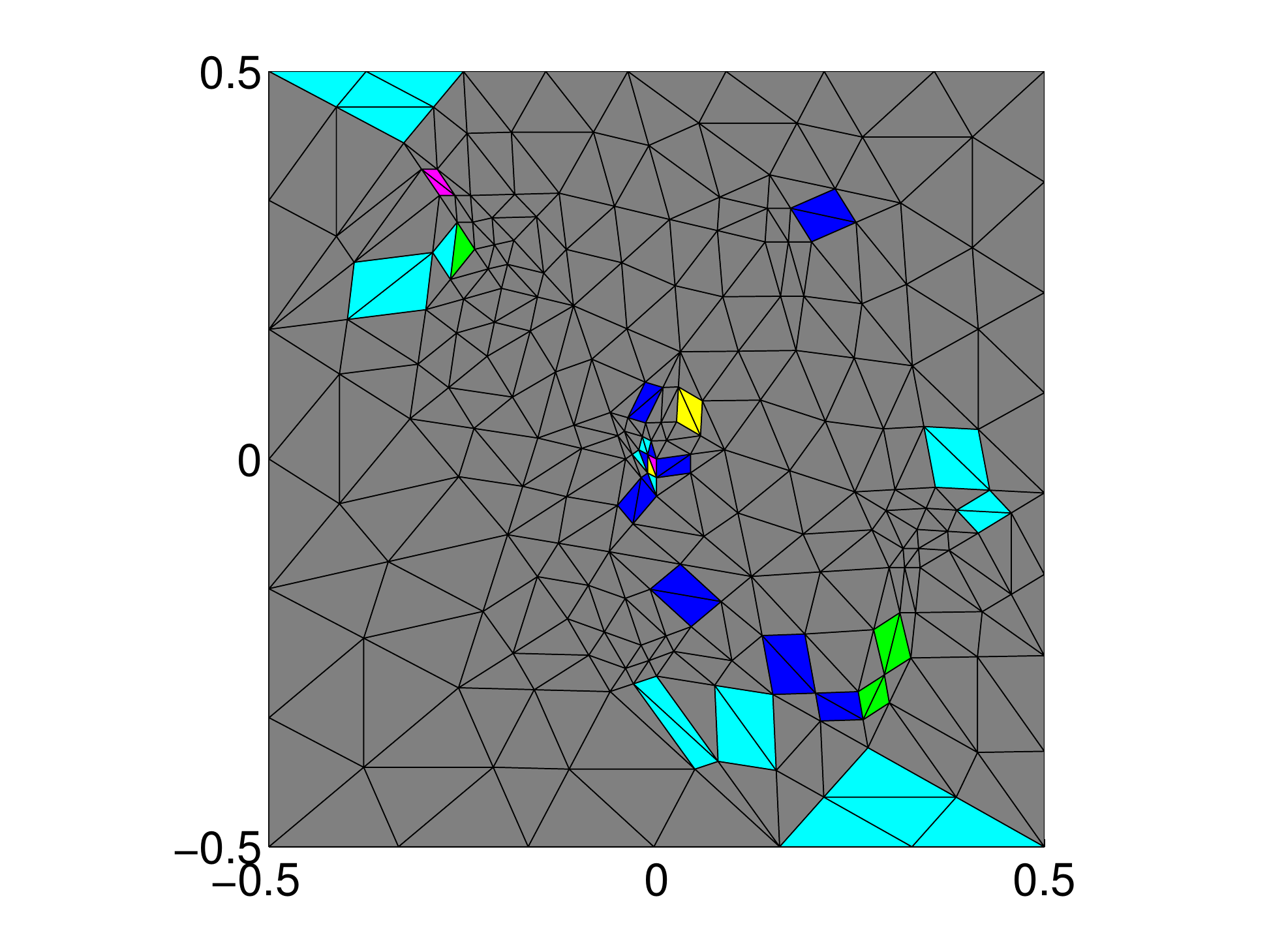} 
\\
\end{tabular}
\caption{The backward error calculated by the local and global minimizations in Sections \ref{sec:analysis} and \ref{sec:design}. The local error is $e = \|\tilde{\fatgamma}_k-\fatgamma\|_2$ on each triangle $T_k\in\calT$.}
\label{fig:BackwardE}
\end{figure}

We compare the finite volume method (FVM), the symmetrized global first exit time method (GFET), the non-negative finite element method (nnFEM), and the method minimizing the backward error (MBE). These methods all produce discretization matrices approximating the Laplacian with only non-negative off-diagonal entries. The experiments are carried out for the mesh in Fig.~\ref{fig:Mesh}. In Fig.~\ref{fig:BackwardE}, the local backward error $e=\|\tilde{\fatgamma}_k-\fatgamma\|_2$ is plotted for all triangles $T_k\in\calT$, calculated by the local and global minimizations in Algorithms \ref{alg:Local} and \ref{alg:Local2} and in \eqref{eq:objfunc}, \eqref{eq:linconstr} and \eqref{eq:objfunc2}, \eqref{eq:linconstr2}.

In Table~\ref{tab:BackwardError} we show the error
\begin{equation}
\eta_2 = \sqrt{\sum_{T_k\in\calT}|T_k|\|\tilde{\fatgamma}_k-\fatgamma\|_2^2/|\Omega|},
\label{eq:GlobalE}
\end{equation}
where $\eta_2$ is either $\eta^L_2$ or $\eta^G_2$ according to \eqref{eq:glopt}. For the global MBE, the matrices $\tilde{\fatgamma}_k$ are first calculated by nnFEM and then used as an initial guess for the MBE optimization. The minima are computed by Matlab's {\tt fmincon} with the active-set algorithm.


\begin{table}[H]
\centering
\begin{tabular}{l| c c}
& $\eta^L_2$ & $\eta^G_2$\\
\hline
FVM & $0.4211$ & $0.2729$ \\
nnFEM & $0.2057$ & $0.1524$ \\
GFET & $0.1963$ & $0.1356$ \\
MBE & $0.0693$ & $0.0690$\\
\end{tabular}
\caption{The global backward errors in \eqref{eq:GlobalE} computed locally $\eta^L_2$ and globally $\eta^G_2$.}
\label{tab:BackwardError}
\end{table}

The local calculation of the backward error is more pessimistic than the global one in the table as expected from \eqref{eq:glopt} but the ranking of the different methods is the same for both $\eta^L_2$ and $\eta^G_2$.
The FVM naturally leading to non-negative jump coefficients causes the largest backward error when used on a poor mesh. A partial explanation to the FVM results may be that the jump coefficients are generated by a different principle than FEM. The fluxes over the element boundaries are approximated in FVM and the method is forced into the framework of FEM. 
On four triangles, a discretization with FVM even leads to a negative definite diffusion matrix $\tilde{\fatgamma}_k$ when calculated locally without the positive definiteness constraint.
The GFET and nnFEM perform comparably for the mesh in Fig.~\ref{fig:Mesh}. Computing $\tfatD$ for GFET is slightly more expensive than setting the negative off-diagonal entries to 0 in nnFEM. However, contrary to the nnFEM the GFET preserves the exit time property of the original diffusion, see \cite{Lotstedt2015}.
The minimization constrained by inequalities to obtain the discretization matrices with MBE improves the introduced backward error substantially. The faster local and slower global minimization algorithms yield similar results for MBE on the mesh in this example.

%

\subsubsection{Forward error}

The relative error in the discrete solution $\|u_h-\tilde{u}_h\|_{L^2}/\|u_h\|_{L^2}$ is computed to verify our analysis in Section~\ref{sec:ErrorEstimate}. Here $\|u_h\|_{L^2}^2=\fatu_h^T\fatM\fatu_h$ with $\fatu_h$ being the vector of the solution in each node and $\fatM$ the lumped mass matrix. The relative error is plotted between the discrete solution $u_h$ with the original discretization matrix $\fatD$ (with negative off-diagonal entries) and the perturbed discrete solution $\tilde{u}_h$ resulting from one of the algorithms generating non-negative off-diagonal entries in $\tfatD$. The system of ordinary differential equations in \eqref{eq:udiscdisturb2} is solved by Matlab's {\tt ode15s}. Fig.~\ref{fig:ForwardE}(a) shows the initial condition
\begin{equation}
u(\fatx,0)=\tanh(20x_1)\tanh(20x_2)+1
\end{equation}
on the mesh in Fig.~\ref{fig:Mesh}.

The forward error $\|u_h(\fatx,t)-\tilde{u}_h(\fatx,t)\|_{L^2}/\|u_h\|_{L^2}$ in space at $t=0.01$ is depicted in Figs.~\ref{fig:ForwardE}(b)-(f).

\begin{figure}[H]
\subfigure[Initial Condition]{\includegraphics[width=0.32\textwidth]{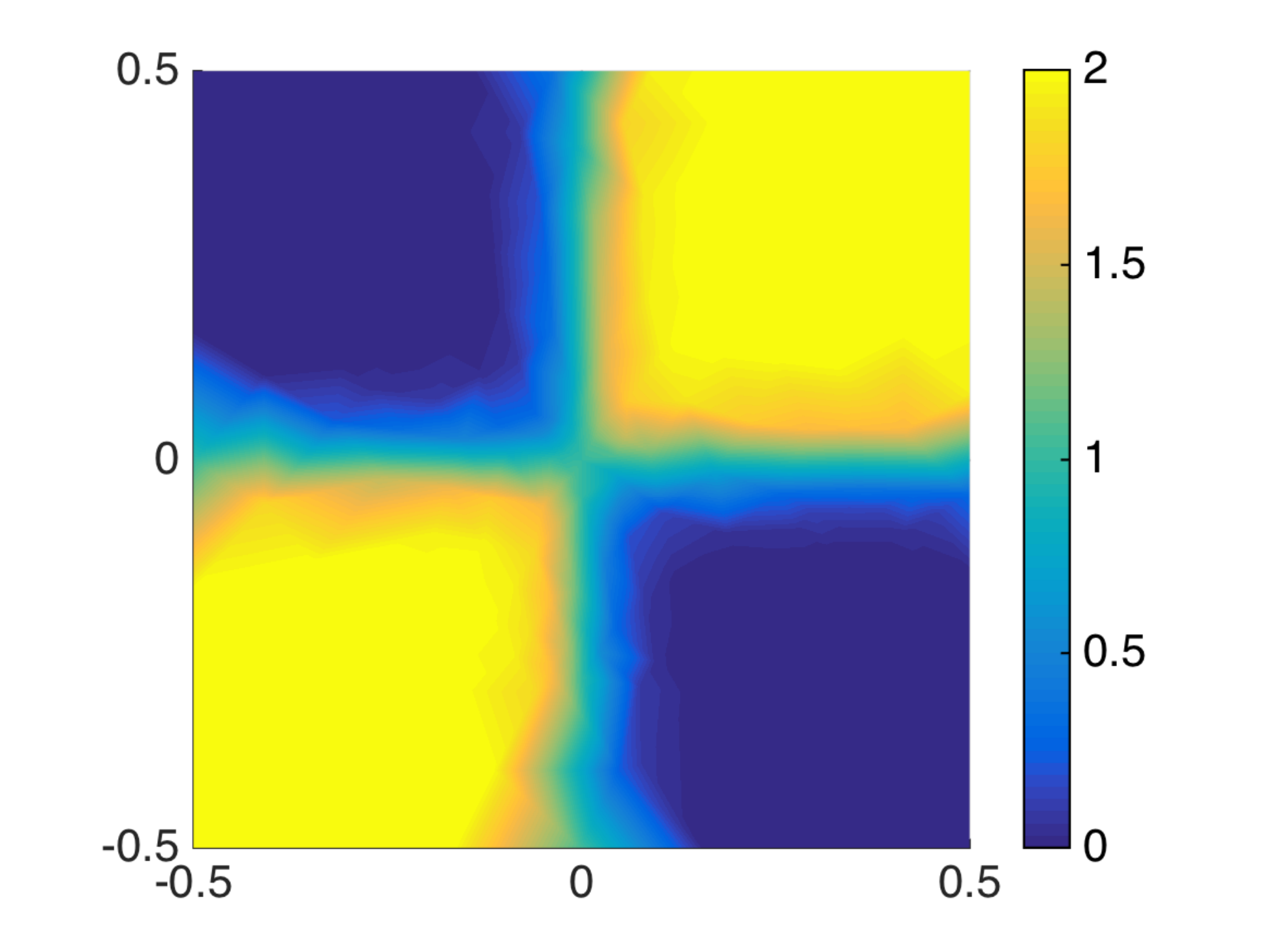}}
\subfigure[GFET]{\includegraphics[width=0.32\textwidth]{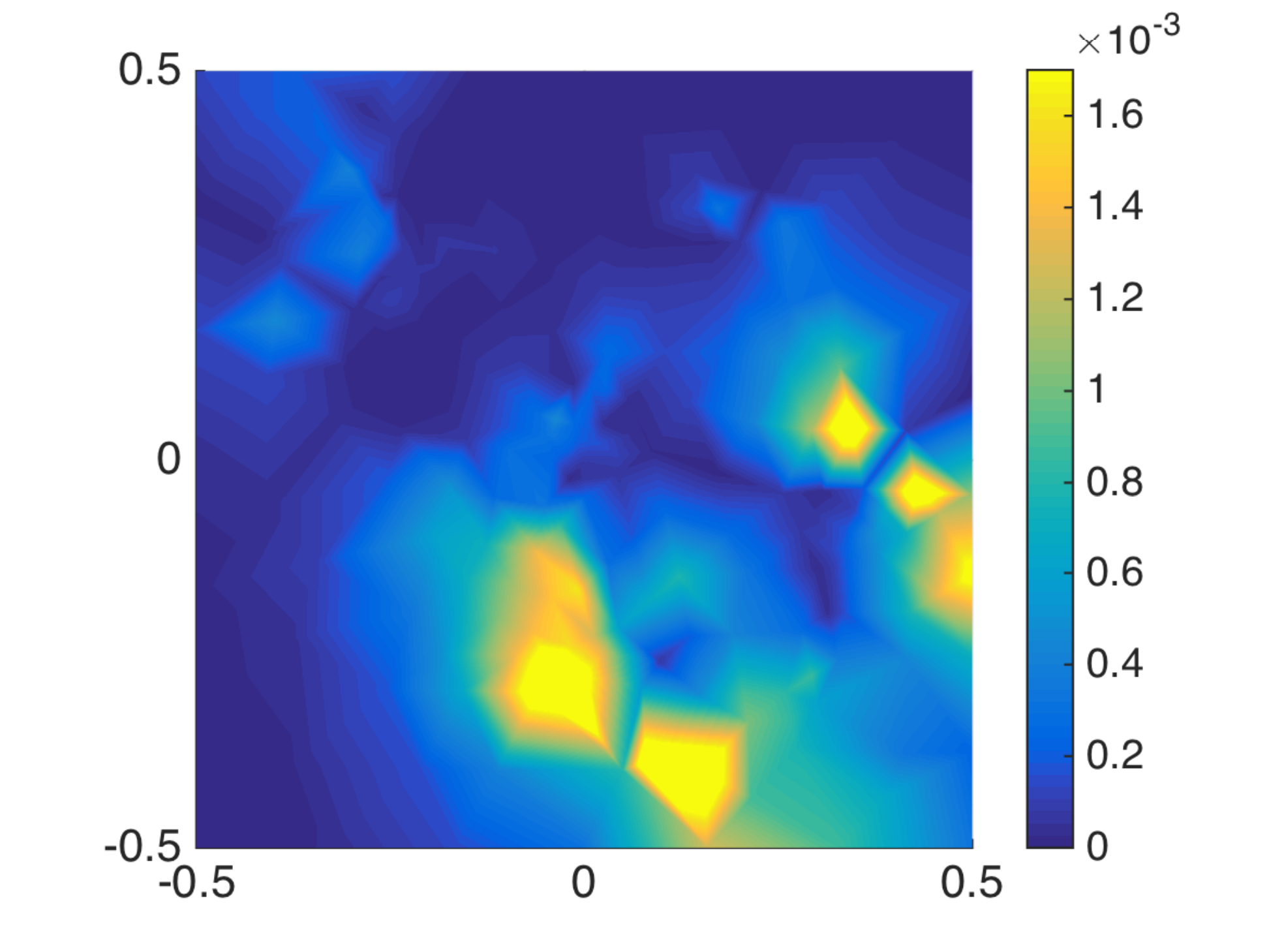}}
\subfigure[nnFEM]{\includegraphics[width=0.32\textwidth]{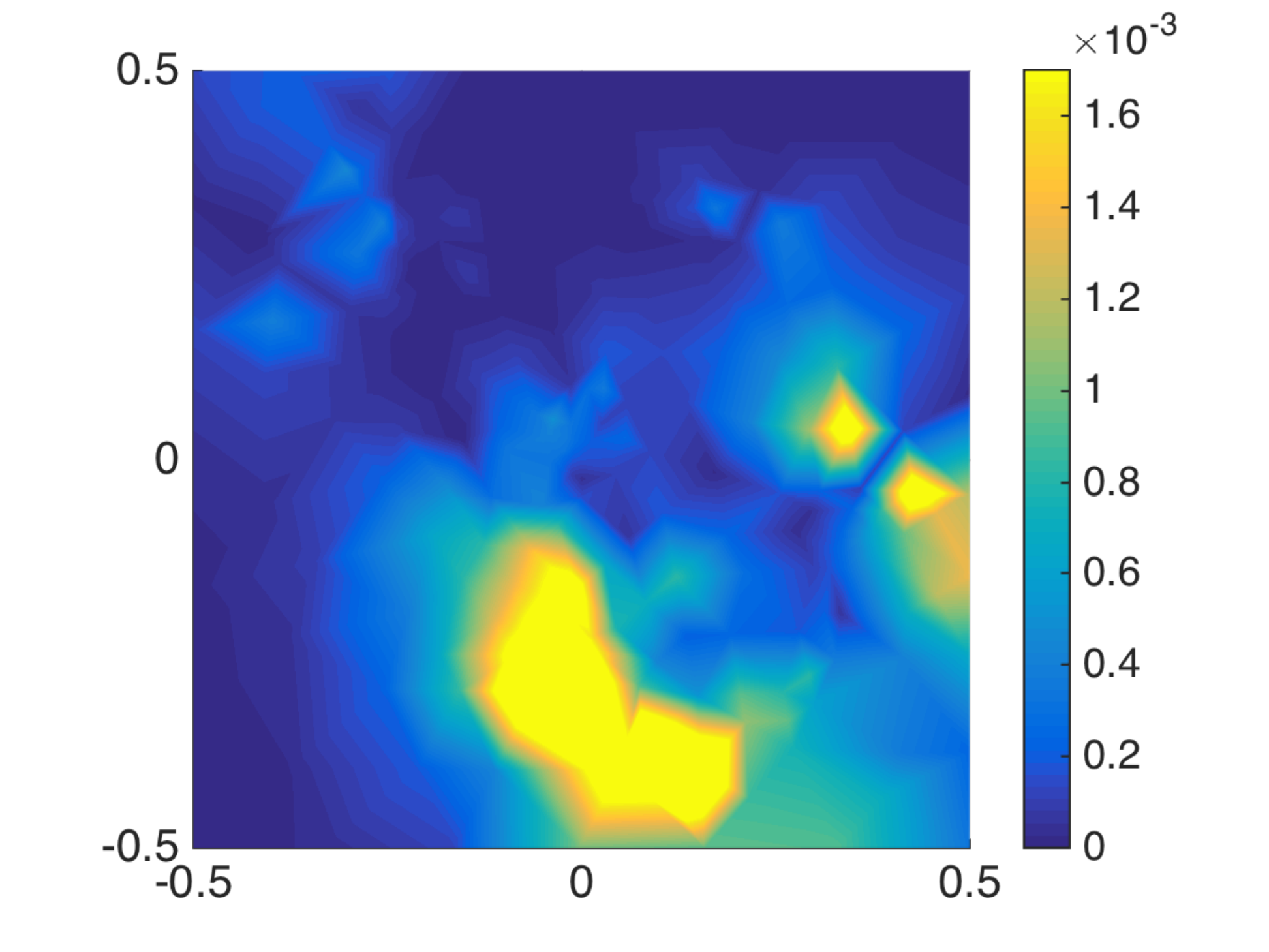}}\\
\subfigure[FVM]{\includegraphics[width=0.32\textwidth]{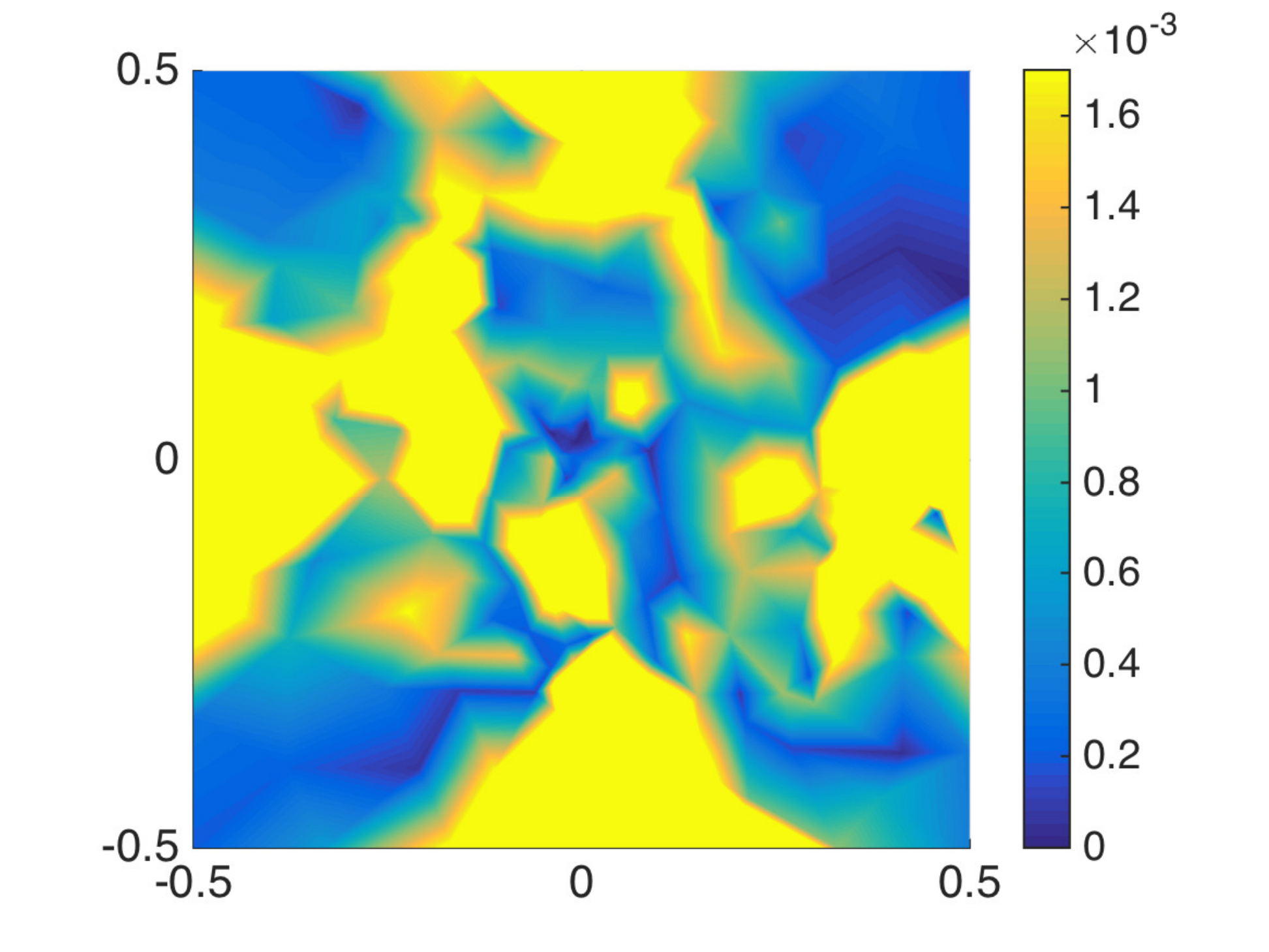}} 
\subfigure[Local MBE]{\includegraphics[width=0.32\textwidth]{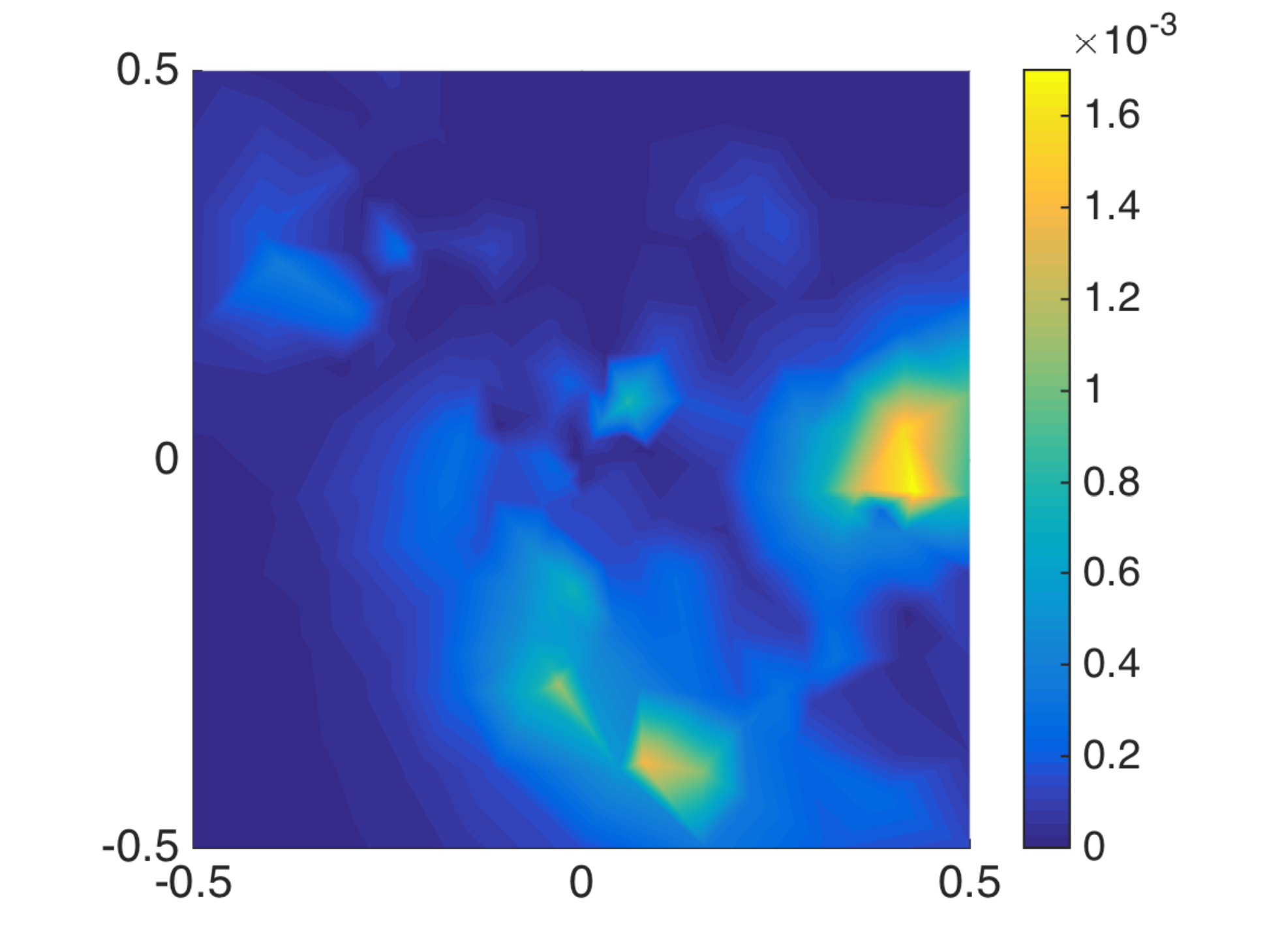} } 
\subfigure[Global MBE]{\includegraphics[width=0.32\textwidth]{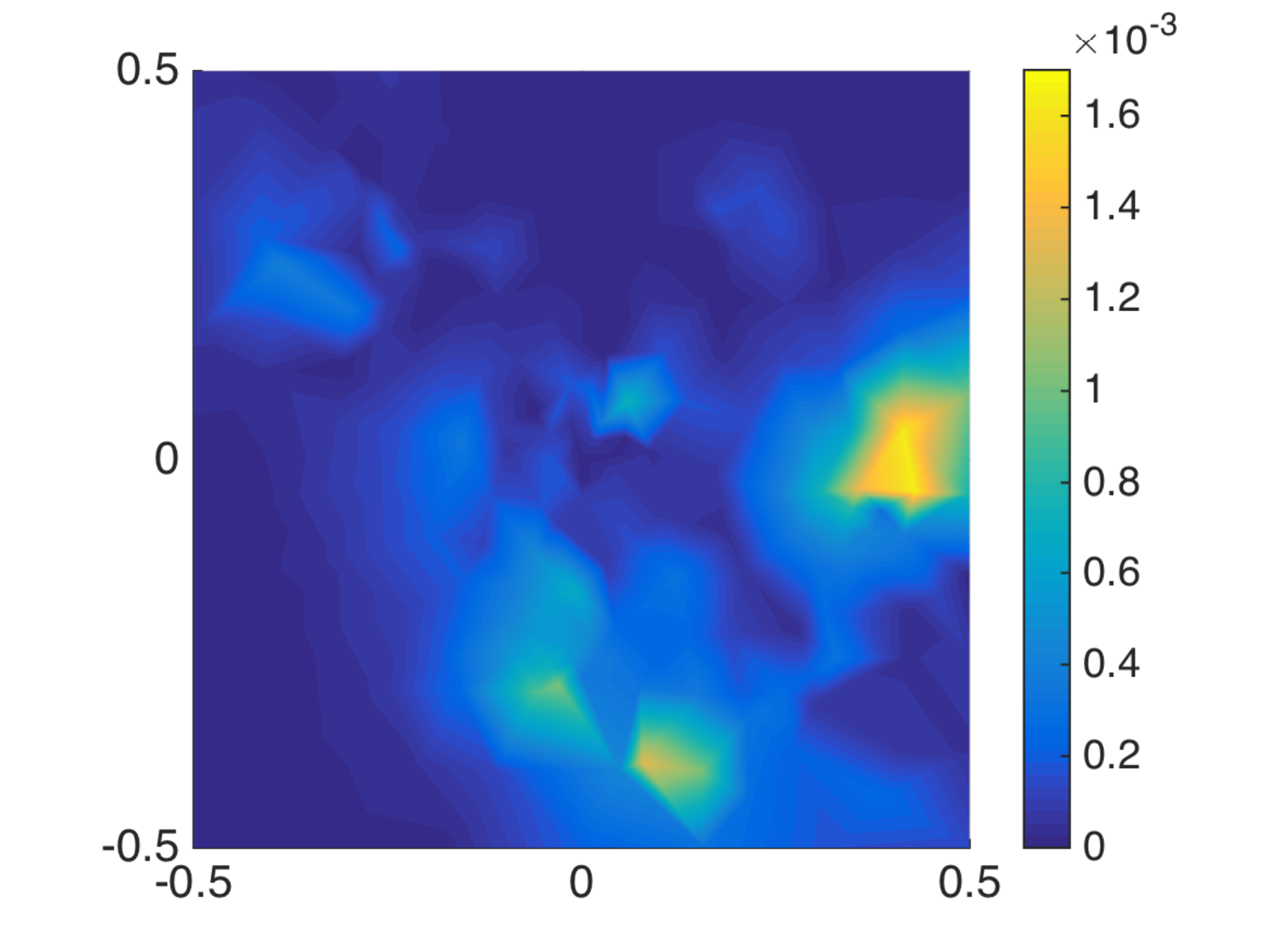}  }
\caption{(a) Initial condition. (b)-(f) Forward error $\|u_h(\fatx_i)-\tilde{u}_h(\fatx_i)\|_{L^2}/\|u_h\|_{L^2}$ for different approximations at each node at $t = 0.01$.}
\label{fig:ForwardE}
\end{figure}

The forward error behaves in the way predicted by the backward error
in Fig.~\ref{fig:BackwardE}, Table~\ref{tab:BackwardError}, and the
stability estimates in Section~\ref{sec:ErrorEstimate}. This is also
confirmed in Fig.~\ref{fig:ForwardETime} where the forward error in
time is displayed in a log-lin scale such that the error for short
times becomes visible.  The unique steady state \eqref{eq:SteadyState}
is reached for large $t$ and the bound $\|u_h-\tilde{u}_h\|_2\leq kt$
with some $k>0$ for the error derived in
Section~\ref{sec:ErrorEstimate} is sharp for small $t$ only.

\begin{figure}[H]
\centering
\includegraphics[width=.5\textwidth]{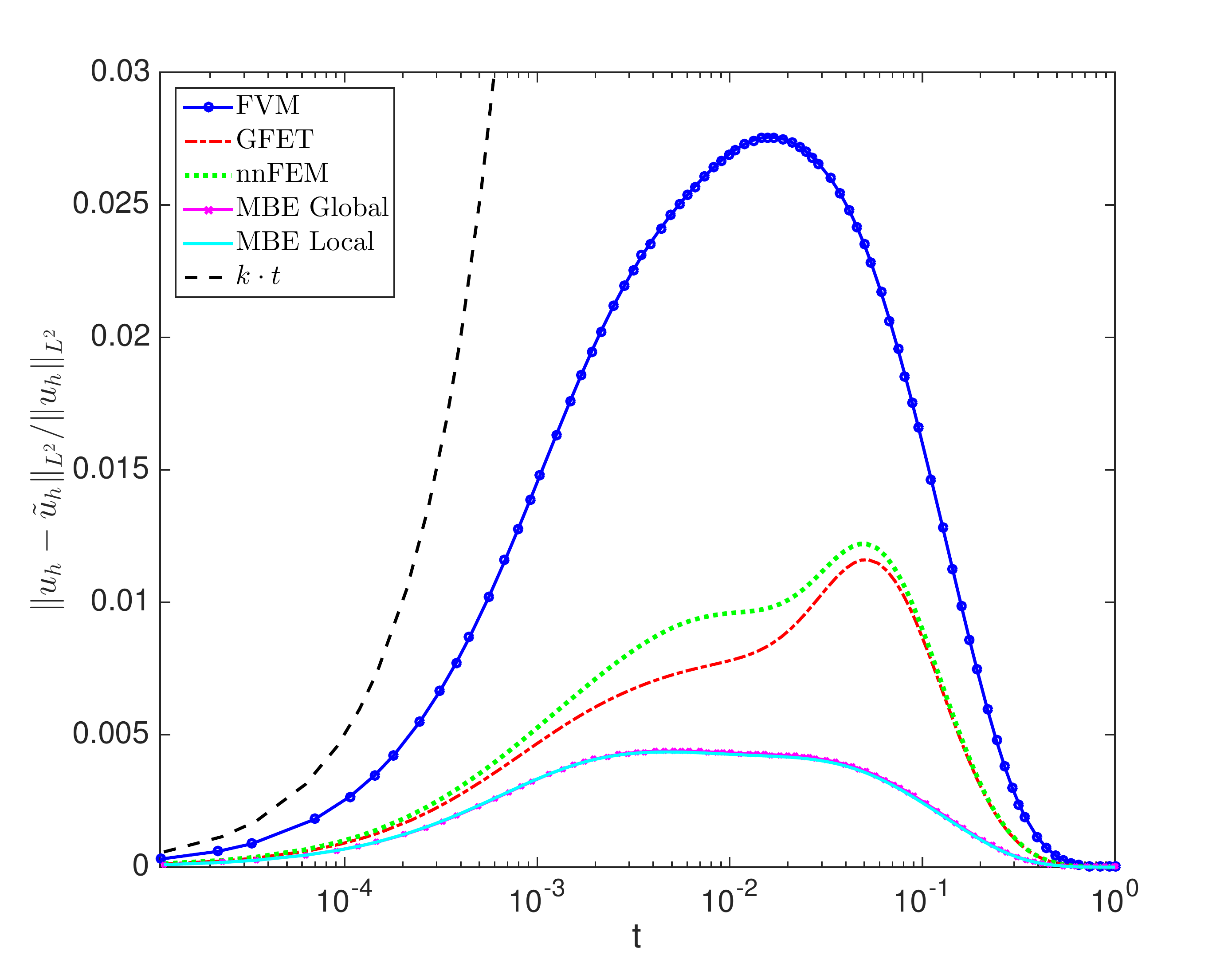}
\caption{Forward relative error $\|u_h-\tilde{u}_h\|_{L^2}/\|u_h\|_{L^2}$ for $0\leq t \leq 1$.}
\label{fig:ForwardETime}
\end{figure}

Comparing the results in Fig.~\ref{fig:ForwardETime} and
Table~\ref{tab:BackwardError} we see that the order between the
methods is the same using the minimization procedure in
Section~\ref{sec:analysis} and the solutions of \eqref{eq:udisc} and
\eqref{eq:udiscdisturb}. The performance of the different methods in
the forward error is correctly predicted by the performance in the
backward error as expected from Section~\ref{sec:ErrorEstimate}. The
MBE is the best and FVM is the worst method but the forward error is
quite small in all methods with a peak for FVM less than three
percent.

\subsubsection{Error in eigenvalues}
In the original equation \eqref{eq:ucont1}, the diffusion is isotropic and the quotient between the eigenvalues in \eqref{eq:coortrans} of $\gamma$ in 2D is $\lambda_1/\lambda_2=1$ and the eigenvectors point in the coordinate directions. The quotient $q=\lambda_{\min}/\lambda_{\max}$ is $g/G$ in Lemma~\ref{lemma:apriori} and for the FVM and the MBE $q$ is found in Fig.~\ref{fig:eigv}. 

\begin{figure}[H]
\subfigure[]{\includegraphics[height=0.2\textheight]{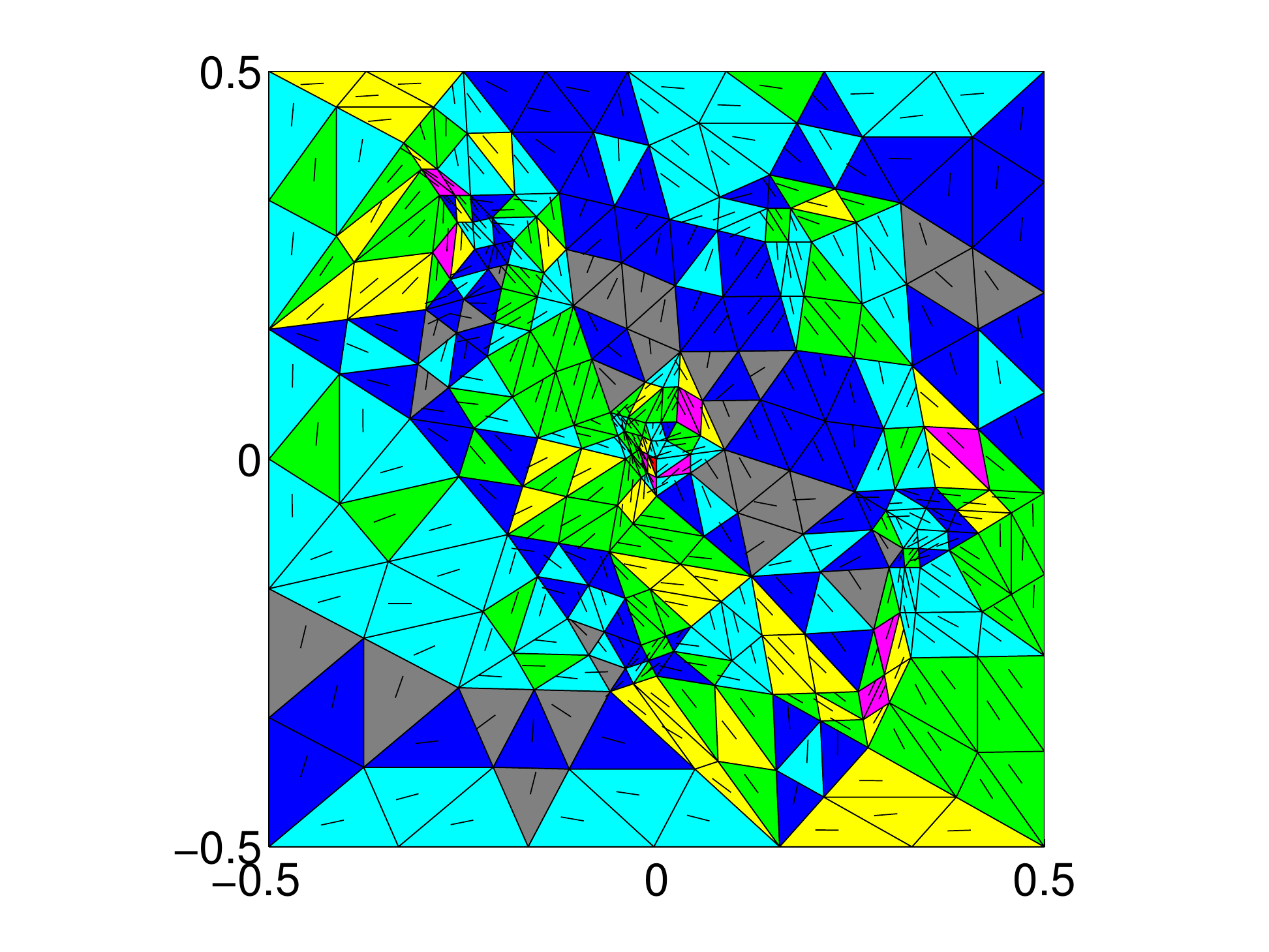}}
\subfigure[]{\includegraphics[height=0.2\textheight]{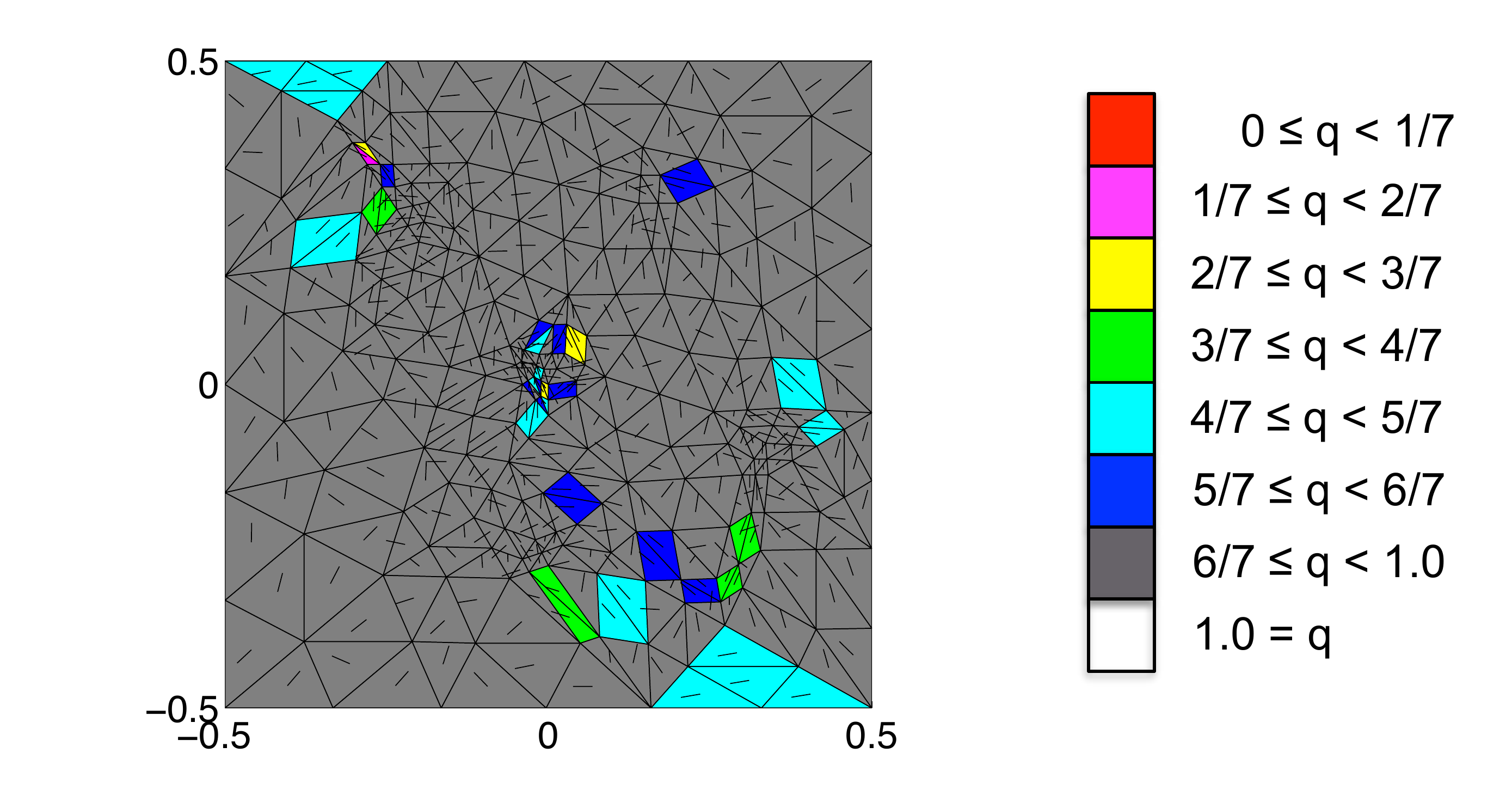}}
\caption{Quotient between the two eigenvalues ${\lambda_{min}}/{\lambda_{max}}$ of $\tilde{\fatgamma}_i$ and the direction of the eigenvector corresponding to the larger eigenvalue. (a) Global FVM. (b) Global MBE.}
\label{fig:eigv}
\end{figure}

Avoiding the negative off-diagonal elements leads to a local anisotropy in the diffusion 
$\tilde{\fatgamma}$ in Fig.~\ref{fig:eigv}. 
The eigenvectors corresponding to the larger eigenvalues are not aligned but point in what appears to be random directions. The effect of the change of the diffusion from $\fatgamma$ to $\tfatgamma$ is randomized over $\Omega$. A global anisotropy is not found here in contrast to what we have in the special regular rhombus mesh in \cite{Lotstedt2015}. All voxels are tilted there in the same direction increasing the diffusion speed in this direction in almost all mesh triangles. A random change of the major axis of diffusion is expected in general in a mesh created by any mesh generator. 

\subsubsection{Alternative methods}
The two alternatives suggested in Section~\ref{sec:alt} - minimizing the difference between $\fatD$ and $\tfatD$ in the $\ell_2$ norm and adding viscosity - are compared to the previous methods.

At a first glance at \eqref{eq:udisc2} and \eqref{eq:udiscdisturb2} it may  seem like minimizing the difference between the matrices $\fatD$ and $\tfatD$ in $\ell_2$ would result in the smallest forward error but this is not the case. Indeed, Fig.~\ref{fig:Alternatives} shows the forward error for this approach together with the best (MBE) and the worst (FVM) methods. Comparing the matrix  deviations in Table~\ref{tab:L2Difference} with the results in Fig.~\ref{fig:Alternatives} reveals that there is only a weak correlation between a small forward error and the closeness of the matrices in the $\ell_2$ norm.
\begin{table}[H]
\centering
\begin{tabular}{l |c}
& $\|\fatD-\tilde{\fatD}\|_2/\|\fatD\|_2$\\
\hline
Viscosity & 0.7856\\
FVM & 0.4520\\
Local MBE & 0.1967 \\
Global MBE & 0.1487\\
nnFEM & 0.1059\\
GFET & 0.0846\\
$\ell_2$ optimal & 0.0696
\end{tabular}
\caption{Relative difference $\|\fatD-\tfatD\|_2/\|\fatD\|_2$ between the discretization matrices.}
\label{tab:L2Difference}
\end{table}

\begin{figure}[H]
\centering
\includegraphics[width =.5\textwidth]{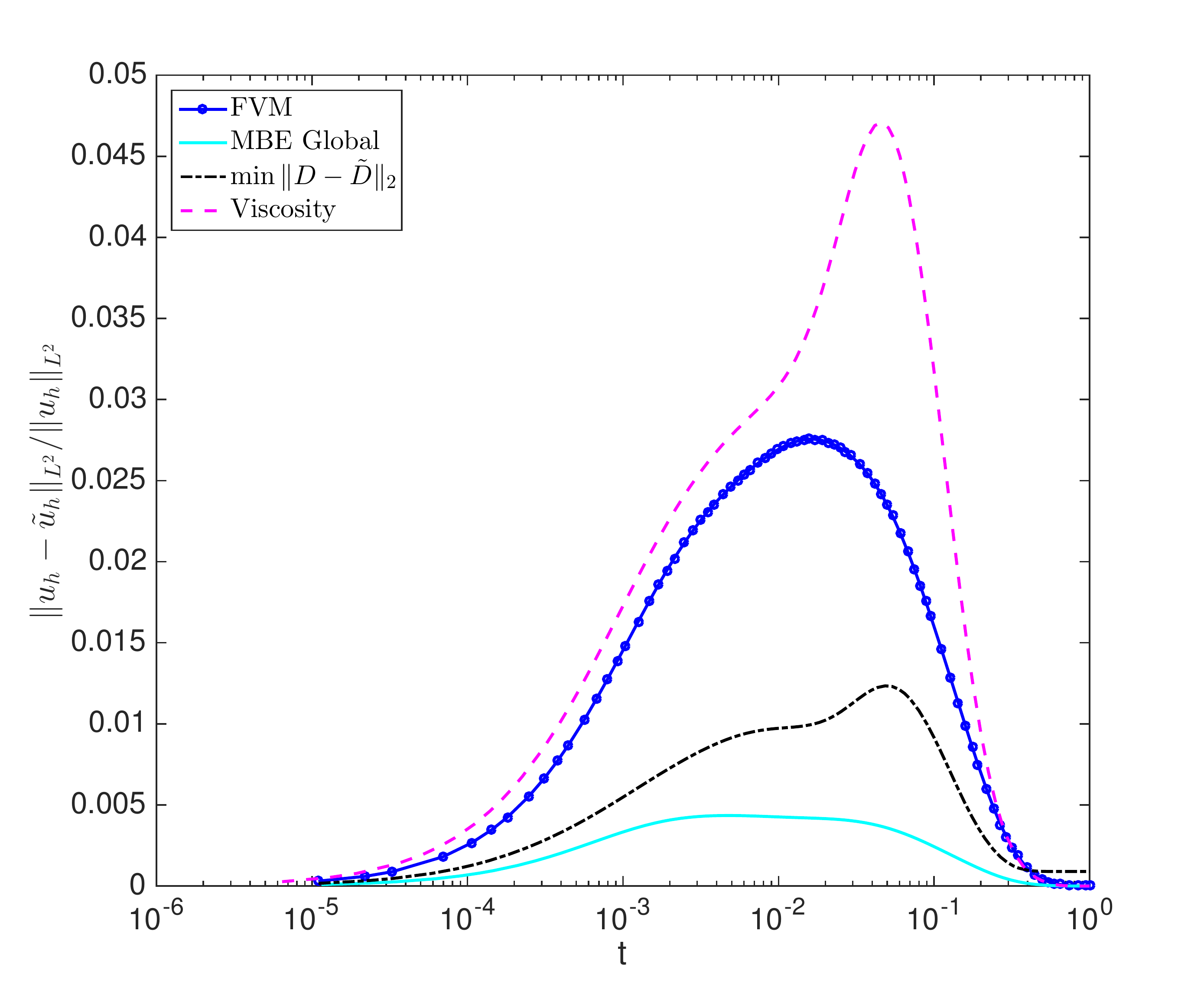} 
\caption{Forward relative error $\|u_h-\tilde{u}_h\|_{L^2}/\|u_h\|_{L^2}$ for $0\leq t \leq 1$. For $\ell_2$-optimization $\eta^G_2=0.1643$. For added viscosity $\eta^G_2=0.6265$.}
\label{fig:Alternatives}
\end{figure}

The $\ell_2$ norm of the difference in the discretization matrices in Table~\ref{tab:L2Difference} neither reflects the behavior of the forward and backward errors nor does the $\tilde{\fatD}$ optimized in the $\ell_2$ norm reproduce the correct steady state in Fig.~\ref{fig:Alternatives}. There is no unique equation and no unique $\tilde{\fatgamma}(x)$ corresponding to a discretization matrix $\tfatD$ in Section~\ref{sec:analysis}. Hence, it is more meaningful to quantify and minimize the error in the solutions $u_h$ and $\tilde{u}_h$ than to compare the discretization matrices representing many different analytical equations.
In Section~\ref{sec:ErrorEstimate} we showed that the $\tfatgamma$ closest to $\fatgamma$ can be used to bound he error in the solution.

We see that adding viscosity to all nodes in the patch leads to adding an unnecessarily large amount of viscosity resulting in a larger error than even the FVM in this case.

\subsubsection{Practical implementation}

%
%


The local backward error is too pessimistic in Table~\ref{tab:BackwardError}, but since only small optimization problems have to be solved involving the local triangles or tetrahedra to an edge, it is faster to compute. Furthermore, the backward error can be reduced by repeatedly applying Algorithm~\ref{alg:Local}. The last solution in one iteration is the initial guess in the next iteration and the edges are traversed in different order in each iteration.

\begin{figure}[H]
\centering
\includegraphics[width=.5\textwidth]{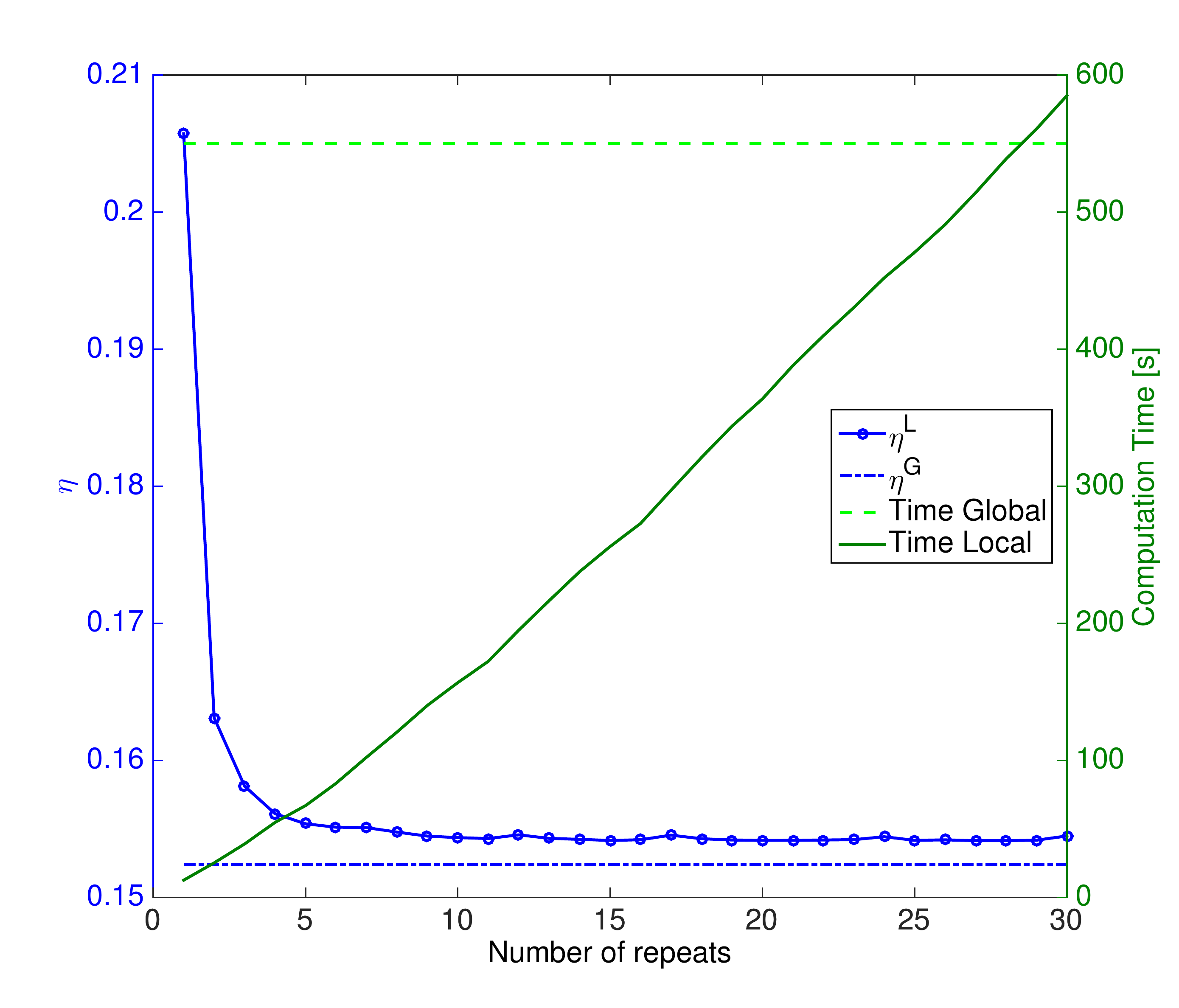} 
\caption{The error $\eta^L_2$ for different number of iterations of Algorithm~\ref{alg:Local} and nnFEM, compared to the error $\eta^G_2$ (blue) and the computing time for $\eta^L_2$ and $\eta^G_2$ (green).}
\label{fig:repeat}
\end{figure}

The effect of repeating Algorithm~\ref{alg:Local} is to spread the local error in each triangle over the domain. The $\eta^L_2$ error is reduced towards the global minimum, when altering the order of the edges in each iteration, but it does not seem to converge to $\eta_2^G$. Repeating Algorithm~\ref{alg:Local} about five times in Fig.~\ref{fig:repeat} is sufficient to achieve an improved backward error at a small increase in computation time.


Minimizing the $\ell_2$ norm becomes prohibitively slow especially in the most expensive algorithm of computing the MBE globally. Therefore, to arrive at a practical implementation we switch to the Frobenius norm \eqref{eq:Frobenius}. Then the minimization problem \eqref{eq:globopt2} has a quadratic objective function in \eqref{eq:globoptFrobenius} and is solved by Matlab's {\tt quadprog} with the interior-point-convex algorithm. The approximate computing times in seconds and the computed backward error for the global MBE are displayed in Table~\ref{tab:TimeFrobenius}.
\begin{table}[H]
\centering
\begin{tabular}{l|c c}
Minimization & Time & $\sqrt{\sum_{T_k\in\calT}|T_k|\|\fatgamma-\tfatgamma_k\|_2^2/|\Omega|}$\\
\hline
 $\|\cdot\|_2$ & 3620 & 0.0690\\
 $\|\cdot\|_F$ & 0.9737 & 0.0804
\end{tabular}
\caption{The computation time in seconds for the global MBE and the resulting global error in the $\ell_2$ norm when minimizing the $\ell_2$ norm and the Frobenius norm in the primal \eqref{eq:globoptFrobenius} problem.}
\label{tab:TimeFrobenius}
\end{table}
Since $\|\fatA\|_2\leq\|\fatA\|_F$ the minimization in the Frobenius norm does not reach the global minimum in the $\ell_2$ norm, but this is compensated for by a substantial speedup.

In Fig.~\ref{fig:Frobenius}, we compare the forward errors for the MBE when minimizing in the Frobenius and $\ell_2$ norms. The error resulting from minimization in the Frobenius norm is not much larger while a reduction of the computing time of more than 3000 is achieved in Table~\ref{tab:TimeFrobenius}.

\begin{figure}[H]
\centering
\includegraphics[width=.5\textwidth]{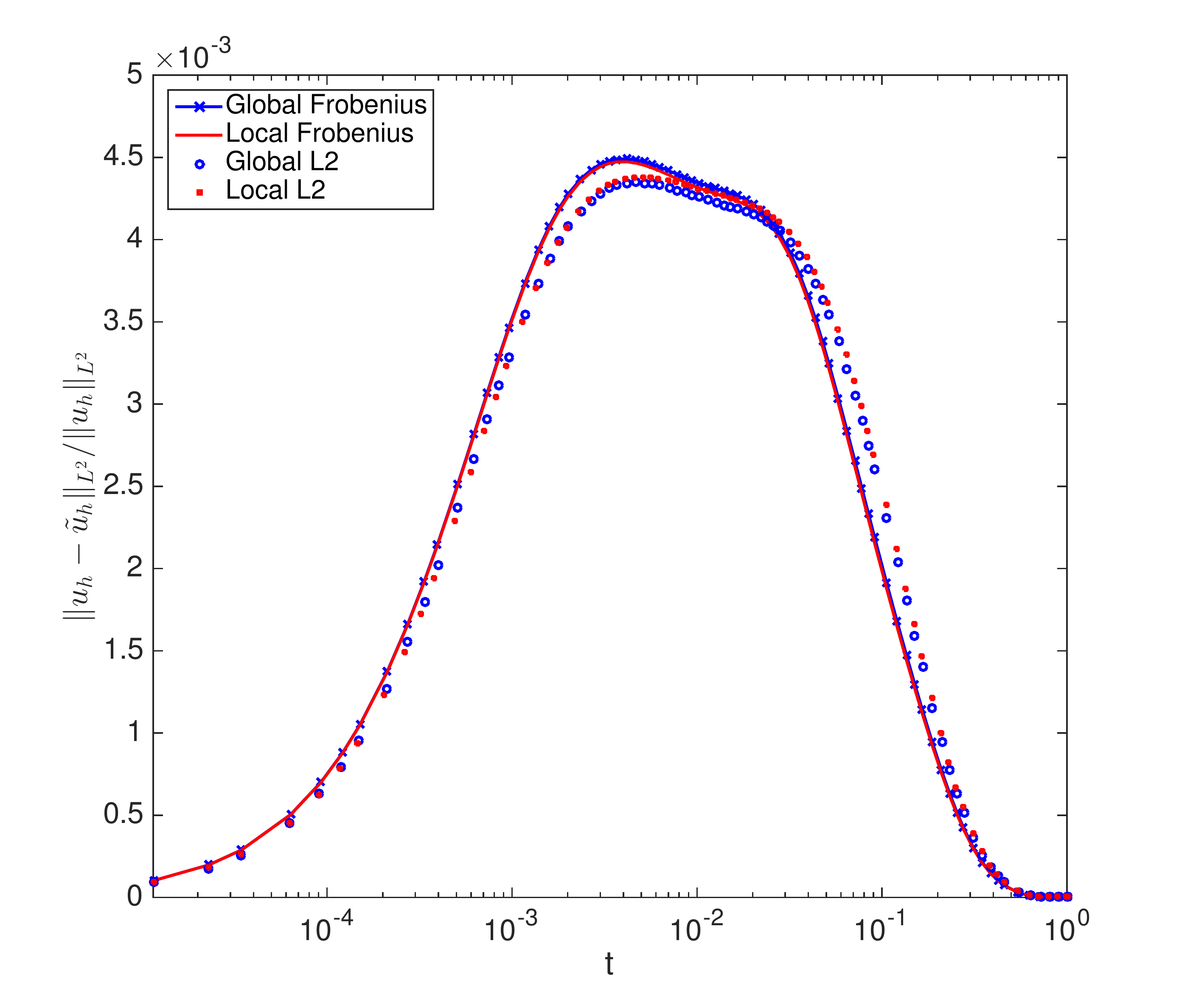} 
\caption{Forward error of the MBE with minimization in Frobenius and $\ell_2$ norms. The errors for the global and local minimizations in the Frobenius norm are indistinguishable.}
\label{fig:Frobenius}
\end{figure}

\subsection{Diffusion in 3D}

Our methods are tested on a more realistic mesh such as those
encountered in systems biology simulations.  A sphere with radius 1 is
discretized into two tetrahedral meshes with 602 and 1660 nodes.
In both meshes about $17$ percent of the edges have a negative jump
propensity with the standard FEM discretization.  Other mesh
generators than the one in COMSOL Multiphysics were tested in
\cite{Kieri} with similar results.

%
%
%

\begin{figure}[H]
\centering
\includegraphics[width=.5\textwidth]{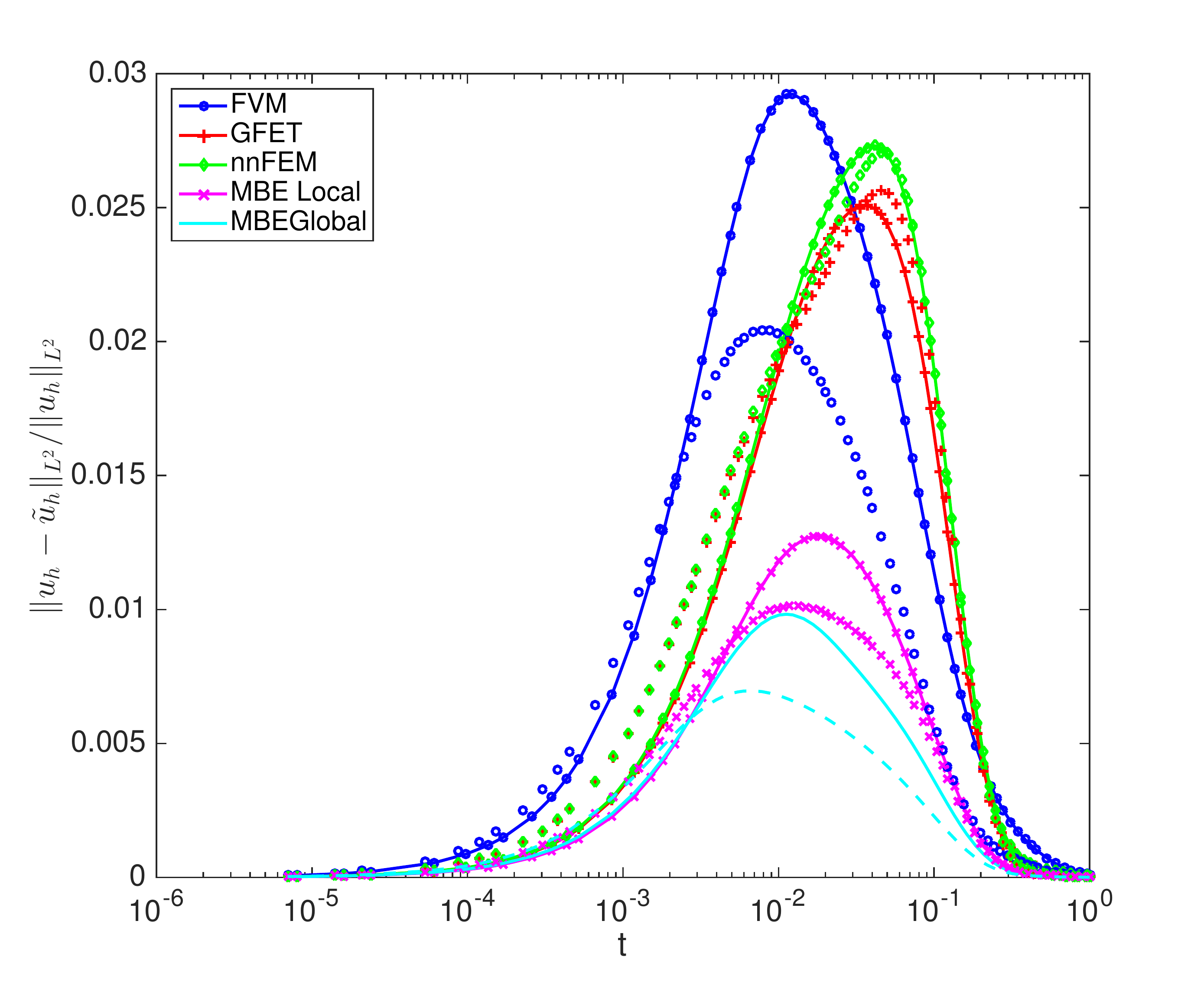}
\caption{Forward error $\|u_h-\tilde{u}_h\|_{L^2}/\|u_h\|_{L^2}$ for the different methods on a mesh with 602 nodes (solid lines) and with 1660 nodes (dashed lines with same markers) discretizing the unit sphere.}
\label{fig:3DForwardE}
\end{figure}

In the following experiments, the global backward error and $\tfatD$
in MBE are computed by minimizing in the Frobenius norm.  The number
of unknown variables in the global optimization problem for the
largest mesh is 51096 in the primal problem
\eqref{eq:globoptFrobenius} and 10599 in the dual problem
\eqref{eq:globoptDual} in Table~\ref{tab:3DBackwardsE}.  The local
optimization problems typically have 24 variables in the primal
problem and 13 in the dual one.  In Fig.~\ref{fig:3DForwardE} and
Table~\ref{tab:3DBackwardsE}, we see that the methods examined in 2D
behave similarly in 3D. The FVM leads to a non-positive definite
diffusion $\tilde{\fatgamma}_k$ in five tetrahedra on the coarse mesh
and twelve on the fine mesh when calculated locally with
Algorithm~\ref{alg:Local} where $\|\cdot\|_2$ is replaced by
$\|\cdot\|_F$. There is a slight difference between the local and
global MBE but the error is in general small for all methods. The
ranking of the methods is the same as in 2D in
Table~\ref{tab:BackwardError} for small $t$. Since the percentage of
negative edges is the same when refining the mesh the nnFEM and GFET
methods do not improve on a finer mesh.
How the backward error of the methods behaves when the mesh is refined
depends not only on the mesh size but also on the shape of the
elements.  For small $t$, the methods perform in forward error on each
mesh as predicted by the respective backward error in
Table~\ref{tab:3DBackwardsE}.

\begin{table}[H]
\centering
\begin{tabular}{l| c c | c c}
& \multicolumn{2}{ |c| }{602} & \multicolumn{2}{ |c }{1660}\\
\hline\hline
& Local & Global & Local & Global\\
\hline
FVM & 0.6415 & 0.4284 & 0.6077 & 0.4227\\
nnFEM & 0.3720 & 0.2898 & 0.3683 & 0.2818\\
GFET & 0.3604 & 0.2618 & 0.3570 & 0.2632\\
MBE & 0.1680 & 0.1593 & 0.1676 & 0.1698
\end{tabular}
\caption{Locally and globally computed errors $\sqrt{\sum_{T_k\in\calT}|T_k|\|\tilde{\fatgamma_k}-\fatgamma\|_2^2/|\Omega|}$ on the coarse (602 nodes) and fine (1660 nodes) meshes.}
\label{tab:3DBackwardsE}
\end{table}

\subsection{Mean first hitting time}
Molecules in biological cells do not only undergo diffusion but also reactions.
In order to measure the error in a way relevant for reaction-diffusion kinetics, we construct a problem that mimics the mean first binding time for two molecules $A$ and $B$ diffusing in a spherical domain in a diffusion limited case. 
The assumption is that the molecules react instantaneously when they are in the same voxel.
We calculate the mean time it takes for molecule $A$ diffusing in $\Omega$ with reflecting boundary conditions at $\pOmega$ to reach a certain node $i$ at $\fatx_i$ where its reaction partner $B$ is located. 
The molecule $A$ is removed when it reaches the node at $\fatx_i$ by introducing a sink at this point. 
This setup models a reaction complex $B$ situated at node $\fatx_i$ transforming our molecule of interest $A$.

Let $p_i(\fatx,t)$ be the probability distribution function of finding $A$ in $\fatx$ at time $t$ when $B$ is at $\fatx_i$ 
and let $\delta$ be the Dirac measure. Then $p_i$ satisfies
\begin{equation}
p_{it}(\fatx,t) = \gamma\Delta p_i(\fatx,t)-k\delta(\fatx-\fatx_i)p_i(\fatx,t)
\label{eq:diffreact}
\end{equation}
with a Neumann boundary condition at $\pOmega$ and a constant
$k>0$. The mean value of the hitting time $\tau_i$ for $A$ to find $B$
is determined for all possible starting positions of $A$ in the mesh.
The initial condition is a uniform distribution of $A$,
$p_i(\fatx, 0)=1/|\Omega|$.  The domain $\Omega$ is the sphere of
radius 1 discretized by the same meshes as in the previous section. A
discrete approximation of \eqref{eq:diffreact} is
\begin{equation}
\fatp_{it} = (\gamma\fatD-\fatK_i)\fatp_i,
\end{equation}
where $\fatp_i^T=(p_{i1}, p_{i2},\ldots,p_{iN})$ and $\fatK_i$ is the
zero matrix except for $K_{ii}=10^9$.

The survival probability $S_i(t)$ and the probability density function $\pi_i(\tau)$ for $\tau_i$ are defined by
\begin{equation}
  S_i(t)=\int_\Omega p_i(\fatx, t)\, d\fatx=P(\tau_i\ge t),\quad\quad \pi_i(t)=-S_{it}.
\end{equation}
The expected value of the hitting time $\tau_i$ can then be calculated by
\begin{equation}
\begin{array}{rl}
\mE[\tau_i]  &= \displaystyle{\int_0^{\infty} \tau \pi_i(\tau)\, d\tau=-\int_0^\infty \tau S_{it}\, d\tau
              =\int_0^\infty \int_\Omega p_i(\fatx, \tau)\, d\fatx d\tau}\\ 
             &\displaystyle{\approx \int_0^{\infty}\sum_{k=1}^N|\calV_k|p_{ik}\, d\tau}.
\end{array}
\label{eq:EHit}
\end{equation}

In Table~\ref{tab:Hitting}, we compare $\mE[\tau_i]$ on the coarse mesh for the original discretization matrix $\fatD$ and the modified discretizations $\tfatD$ described in Sections~\ref{sec:mesodiff} and \ref{sec:design}. A sink is placed at one node $i$ in the mesh. 
Since many interesting reactions in cells occur in reaction complexes bound to the membrane or in the nucleus we especially investigate the time it takes for $A$ to either find $B$ at a boundary node or at the node closest to the center. The average of $\mE[\tau_i]$ over $i$ is first computed when the sink and $B$ are at any node and at any node on the boundary. Then the sink is at the node closest to the center. Finally, $\mE[\tau_i]$ is computed with stochastic simulation employing Algorithm 1. The time for $A$ to reach $B$ at the center is recorded for each trajectory and the average is taken over $10^5$ realizations. The inital position of $A$ is sampled from a uniform distribution. With $N$ being the total number of nodes, $N_B$ the number of boundary nodes, $M$ the number of trajectories, and $\tau^m_c$ the hitting time of trajectory $m$, the quantities in the table are
\begin{equation}
\begin{array}{lll}\label{eq:Edef}
\displaystyle{E_{\text{All}}=\frac{1}{N}\sum_{i=1}^N \mE[\tau_i],\quad E_{\text{Bnd}}=\frac{1}{N_B}\sum_{\fatx_i\in\pOmega} \mE[\tau_i],\quad E_{\text{Cdet}}= \mE[\tau_c],}\\
\displaystyle{E_{\text{Cstoch}}=\frac{1}{M}\sum_{m=1}^M \tau_c^m,\quad E_{\text{Std}}^2=\frac{1}{N_B-1}\sum_{\fatx_i\in\pOmega} (\mE[\tau_i]-E_{\text{Bnd}})^2.}
\end{array}
\end{equation}

The results with standard FEM are found in the top row in the table as
reference values. The FEM values are second order accurate and
converge to the analytical values of the original diffusion equation
when the mesh size is reduced. Since the FEM stiffness matrix has
negative off-diagonal elements, stochastic simulation with its jump
coefficients is impossible. In Fig.~\ref{fig:Radius}, we illustrate
the results in Table~\ref{tab:Hitting} by plotting the expected time
to reach a node $\fatx_i$ as a function of its radial position. We
average the expected exit times at all nodes in shells of the sphere
of width 0.1.

\begin{table}[H]
\centering
\begin{tabular}{l | rrrr}
& $E_{\text{All}}$ & $E_{\text{Bnd}}$ & $E_{\text{Cdet}}$ & $E_{\text{Cstoch}}$ \\
\hline
\hline
FEM & 8.0413 & 11.6860 & 4.9211 & N/A \\
\hline
FVM & 8.8255 & 12.7055 & 5.9304 & 5.9733 \\
GFET & 7.7308 & 11.4102 & 3.6250 & 3.6424 \\
nnFEM & 7.4794 & 10.8482 & 4.4648 & 4.4707 \\
MBE & 8.2944 & 12.0863 & 5.3001 & 5.3323 \\
\end{tabular}
\caption{Averages of the expected first hitting time $\mE[\tau_i]$ defined in \eqref{eq:Edef} for different methods in Columns 2-5 on the mesh with 602 nodes.}
\label{tab:Hitting}
\end{table}

\begin{figure}[H]
\centering
\includegraphics[width=.5\textwidth]{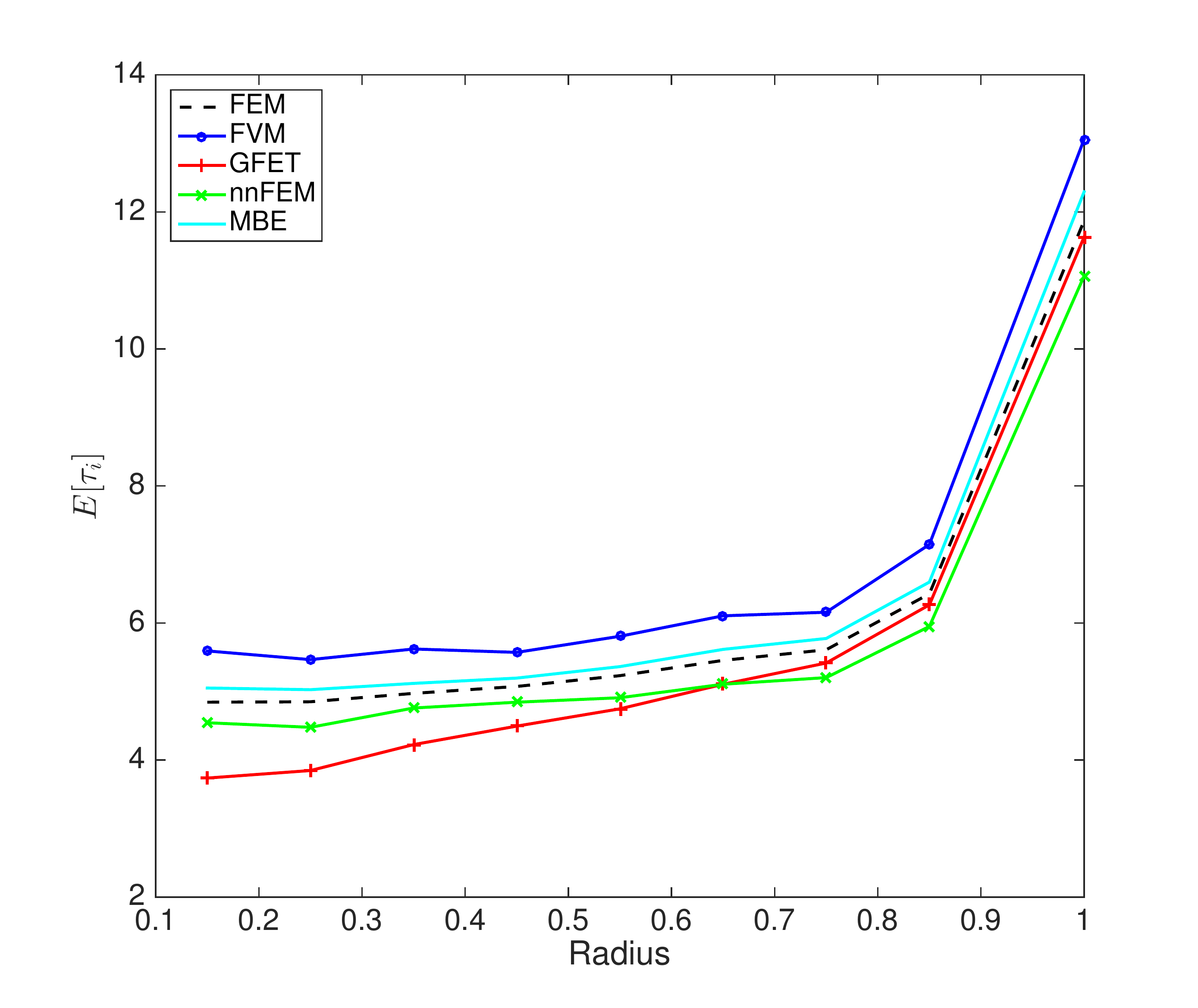} 
\caption{Expected exit times as function of the radial position.}
\label{fig:Radius}
\end{figure}

The MBE is the superior method both in the average over all nodes and
to reach the center when compared to the FEM in
Table~\ref{tab:Hitting} and Fig.~\ref{fig:Radius}. The GFET which was
designed in \cite{Lotstedt2015} to be accurate for the global first exit
time - meaning the first hitting time of a boundary node - performs
best for the boundary nodes. The GFET is, however, not able to compute
the time to reach the center of the cell very accurately. This
shortcoming of the method was discussed in \cite{Lotstedt2015}. The FVM is
too slow with a longer time to reach the sinks than FEM and GFET and
nnFEM are a little too fast corresponding to a global diffusion
coefficient larger than $\gamma$. This tendency was noted also in
\cite{Lotstedt2015}.  The results of the stochastic simulations in Column 5
for the central node are close to the deterministic values of
\eqref{eq:EHit} in Column 4 as expected for a large $M$.

The standard deviation of the mean time it takes to reach a boundary
node measures the variation between different nodes at the boundary
and should be small (ideally 0) for a good mesh and an accurate
discretization. In Table~\ref{tab:STD}, we compare the expected time
to reach a boundary node and its standard deviation \eqref{eq:Edef}
for the different methods on the two meshes.

\begin{table}[H]
\centering
\begin{tabular}{l| c c | c c}
& \multicolumn{2}{ |c| }{602} & \multicolumn{2}{ |c }{1660}\\
\hline
& $E_{\text{Bnd}}$ & $E_{\text{Std}}$/$E_{\text{Bnd}}$ & $E_{\text{Bnd}}$ & $E_{\text{Std}}$/$E_{\text{Bnd}}$\\
\hline\hline
FEM & 11.6860 & 0.0876 & 15.8015 & 0.0808\\
\hline
FVM & 12.7055 & 0.1557 & 16.7264 & 0.1386\\
GFET & 11.4102& 0.1170 & 15.4179 & 0.1060\\
nnFEM & 10.8482 & 0.1171 & 14.6743 & 0.1060\\
MBE & 12.0863 & 0.0908 & 16.2679 & 0.0853
\end{tabular}
\caption{Expected time to reach a boundary node and its standard deviation on a mesh with 602 nodes (left) and a fine mesh with 1660 nodes (right).}
\label{tab:STD}
\end{table}

The standard deviation of MBE is close to that of FEM demonstrating
that the anisotropy in the diffusion introduced by MBE has a small
impact compared to the accuracy effects of the discretization and the
mesh. The small difference in $E_{\text{Std}}/E_{\text{Bnd}}$ between
MBE and FEM is most likely explained by the random directions of
maximum and minimum diffusion as in Fig.~\ref{fig:eigv}.
The relative standard deviation is reduced slightly when the mesh is
refined but $E_{\text{Bnd}}$ has not yet converged.  The
$E_{\text{Bnd}}$ closest to the FEM value is obtained by GFET. When
simulating a signal being transmitted inside the cell it is
advantageous to use GFET if the important reactions occur on the
membrane. If the signal on the other hand is travelling inside the
cytoplasm and reacting there, then the MBE results in the most
accurate transmission time in Table~\ref{tab:Hitting}.


\section{Conclusion}
\label{sec:concl}

For the discrete stochastic simulation of diffusion in systems
biology, we need jump propensities for the molecules in the discrete
space model. These propensities are chosen as the off-diagonal
elements of the discretization matrix obtained by a numerical
approximation of the Laplacian. The jump coefficients have to be
non-negative. For unstructured meshes, non-negative off-diagonal
elements cannot be guaranteed with a discretization matrix assembled
by a standard finite element method (FEM) but there exist different
approaches to change this discretization matrix to fulfill the
non-negativity condition. As a result of this change, a diffusion
equation with an altered diffusion is approximated.

We first present a method to analyze these existing methods producing
non-negative jump propensities on an unstructured mesh of poor
quality. The difference between the solution to the original and the
perturbed diffusion equations is bounded by the difference in the
diffusion coefficients. Then the perturbed diffusion is retrieved by
backward analysis. This leads us to the derivation of a new algorithm
creating a discretization matrix based on FEM, minimizing the backward
error.

We show in numerical experiments that the finite volume method (FVM)
to compute a non-negative discretization incurs high forward and
backward errors on our meshes. Our previously proposed methods of
eliminating the negative entries in the finite element matrix (nnFEM)
and satisfying the global first exit time constraint (GFET) perform
comparably. The new method to generate jump coefficients on a given
mesh proposed in this paper results in a considerably smaller error on
both an artificial mesh in 2D and a realistic mesh in 3D.

The average of the first hitting time obtained by stochastic
simulations with non-negative jump coefficients is close to the
solution of a deterministic equation with a modified diffusion as
expected. The accuracy of this average compared to the exact
analytical values not only depends on the number of trajectories in
the Monte Carlo simulation but also on the mesh size and the mesh
quality. In general, with the MBE the backward and forward errors are
small and the mean hitting time to any node is well approximated.  The
errors in the stochastic diffusion simulation are of the same order as
the errors in biological measurements \cite{PLUE13, RPGBC}.
Furthermore, there is a variation in the diffusion constant across the
cell in measurements in \cite{RaLaGra14} comparable to the variation
of the space dependent $\tfatgamma$ of the perturbed diffusion
equation \eqref{eq:ucontdisturb2} determined by our algorithms.

Since the off-diagonal elements are non-negative, the FEM
discretization of the equation with the modified diffusion satisfies
the sufficient conditions for the discrete maximum principle for the
FEM solution. The discrete maximum principle being satisfied for the
original diffusion on any mesh, seems to be possible only with a
non-linear scheme as in \cite{BuErn}. Then the stiffness matrix
$\fatS$ is reassembled in every time step. Our scheme is linear with a
constant $\fatS$ but for a modified diffusion.


\section*{Acknowledgment}

This work has been supported by the Swedish Research Council with grant number 621-2011-3148, the
UPMARC Linnaeus center of Excellence, the Swedish strategic research programme eSSENCE, and the NIH grant for StochSS
with number 1R01EB014877-01. We have had fruitful discussions with
Murtazo Nazarov and Stefano Serra-Capizzano concerning parts of this
work.

\bibliographystyle{plain}
\bibliography{mesomicro}


\end{document}